\documentclass[11pt]{amsart}
\usepackage{latexsym,amssymb,amsfonts,amsmath,graphicx,fullpage,url}
\usepackage[vcentermath,enableskew]{youngtab}
\usepackage{color}
\usepackage[all]{xy}
\usepackage{verbatim}
\usepackage{makecell}
\usepackage{gensymb}
\usepackage{caption}
\usepackage{subcaption}
\usepackage{ytableau}
\usepackage{MnSymbol}
\usepackage{mathtools}
\newtheorem{theorem}{Theorem}[section]
\newtheorem{proposition}[theorem]{Proposition}
\newtheorem{lemma}[theorem]{Lemma}

\newtheorem{corollary}[theorem]{Corollary}
\theoremstyle{definition}
\newtheorem{remark}[theorem]{Remark}
\newtheorem{definition}[theorem]{Definition}
\newtheorem{example}[theorem]{Example}

\newcommand{\orb}{\mathcal{O}}
\newcommand{\linex}{\mathcal{L}}
\DeclareMathOperator*{\row}{\rm Row}
\DeclareMathOperator{\kpro}{{\it K}-Pro}
\newcommand{\inc}[2]{\mathrm{Inc}^{#1}(#2)}
\newcommand{\prodchains}{J([a] \times [b] \times [c])}
\newcommand{\prodchainstwo}{J([2] \times [b] \times [c])}

\newcommand{\recomb}[2]{\Delta_{#1}^{#2} I}
\newcommand{\genrecomb}[2]{\Delta_{#1}^{#2} I}
\newcommand{\prorecomb}[2]{recombination}

\newcommand{\new}[1]{\textcolor{black}{#1}}

\title{Homomesy in products of three chains and multidimensional recombination}
\author{Corey Vorland}

\begin{document}

\maketitle

\begin{abstract}
	J.~Propp and T.~Roby isolated a phenomenon in which a statistic on a set has the same average value over any orbit as its global average, naming it homomesy. They proved that the cardinality statistic on order ideals of the product of two chains poset under rowmotion exhibits homomesy. In this paper, we prove an analogous result in the case of the product of three chains where one chain has two elements. In order to prove this result, we generalize from two to $n$ dimensions the recombination technique that D.~Einstein and Propp developed to study homomesy. We see that our main homomesy result does not fully generalize to an arbitrary product of three chains, nor to larger products of chains; however, we have a partial generalization to an arbitrary product of three chains. Additional corollaries include refined homomesy results in the product of three chains and a new result on increasing tableaux. We conclude with a generalization of recombination to any ranked poset and a homomesy result for the Type B minuscule poset cross a two element chain.
\end{abstract}

\section{Introduction}

Homomesy is a surprisingly ubiquitous phenomenon, isolated by J.~Propp and T.~Roby \cite{PR2015}, that occurs when a statistic on a combinatorial set has the same average value over orbits of that action as its global average. Homomesy has been found in actions on tableaux \cite{BPS2016, PR2015}, actions on binary strings \cite{Roby2016}, rotations on permutation matrices \cite{Roby2016}, certain products of toggles on noncrossing partitions \cite{EFGJMPR2016}, Suter's action on Young diagrams \cite{PR2015} (with proof due to D.~Einstein), linear maps acting on vector spaces \cite{PR2015}, a phase-shift action on simple harmonic motion \cite{PR2015}, and others. A motivating instance of this phenomenon is the action of \emph{rowmotion} on order ideals of a poset. Rowmotion on an order ideal is defined as the order ideal generated by the minimal poset elements that are not in the order ideal; this action has generated significant interest in recent algebraic combinatorics, giving rise to many beautiful results \cite{CF1995, DPS2017, EP2014, PR2015, SW2012}. For a survey of recent homomesy results, see \cite{Roby2016}; for an introduction to dynamical algebraic combinatorics, including rowmotion, see \cite{Striker2017}. Our initial motivation for this paper was Propp and Roby's result that the cardinality statistic on order ideals of the product of two chains poset $[a]\times[b]$ under rowmotion exhibits homomesy \cite{PR2015}. D.~Rush and K.~Wang generalized this result by showing all \emph{minuscule posets} exhibit homomesy under rowmotion using the cardinality statistic \cite{RW2015}; the product of chains is the Type A case of this result.

In this paper, we investigate homomesy in the product of three chains, or equivalently, a type A minuscule poset cross a chain. More specifically, we show order ideals of $[2]\times[b]\times[c]$ exhibit homomesy under promotion with cardinality statistic. However, we observe such a homomesy result does not hold for a general product of three chains. We also obtain a homomesy result on order ideals of a type B minuscule poset cross a chain of size two. To prove these results, we generalize the \emph{recombination} technique of Einstein and Propp \cite{EP2014} from two to $n$ dimensions. Recombination is a tool that Einstein and Propp developed to translate homomesy results between rowmotion and a related action called \emph{promotion} by J.~Striker and N.~Williams in \cite{SW2012}. Einstein and Propp showed recombination gives an equivariant bijection between order ideals of $[a] \times [b]$ under rowmotion and order ideals of $[a] \times [b]$ under promotion. Using a different method, Striker and Williams showed that there is an equivariant bijection between order ideals of any ranked poset under promotion and under rowmotion. This means that the orbit structure is the same under rowmotion and promotion, so if we want to study the orbits of rowmotion, we could instead study the orbits of promotion, or vice versa. K.~Dilks, O.~Pechenik, and Striker \cite{DPS2017} generalized promotion to higher dimensions. Furthermore, they showed that for a given poset, there is an equivariant bijection between any of the multidimensional promotions they defined. Underlying all these results is the \emph{toggle group} of P.~Cameron and D.~Fon-der-Flaass \cite{CF1995}, who provided access to the tools of group theory by exhibiting rowmotion as a toggle group action.

\textbf{Our first main theorem}, Theorem \ref{thm:mainhom}, says that the order ideals of $[2] \times [b] \times [c]$ exhibit homomesy with average value $bc$ under promotion when using the cardinality statistic. To prove this theorem, we generalize the recombination result of Einstein and Propp from a product of chains $[a] \times [b]$ to a product of chains $[a_1] \times \dots \times [a_n]$ in our second main theorem, Theorem \ref{thm:genrecomb}. As part of proving our first main theorem, we also translate a homomesy result on increasing tableaux of shape $2 \times b$ under \emph{K-promotion} with statistic box entry summation to order ideals of $[2] \times [b] \times [c]$ under a specific promotion with cardinality statistic. We also prove the following additional results. In Propositions \ref{prop:334} and \ref{prop:22222}, we show that our homomesy result does not generalize to order ideals of $[a]\times [b] \times [c]$ or order ideals of $[2] \times \dots \times [2]$ under promotion with cardinality statistic. Although our result does not generalize fully to products of three chains, using Pechenik's homomesy result on the \emph{frame} of an increasing tableaux of shape $a \times b$ with statistic box entry summation \cite{Pechenik2017}, in Corollary \ref{cor:gen3chain}, we establish homomesy on $[a]\times [b] \times [c]$ under promotion with cardinality statistic on the ``outside" of the poset. Additionally, Corollaries \ref{cor:rotsym} and \ref{cor:gen3chain} include refinements of our main homomesy result and this partial generalization, respectively, where we consider the cardinality statistic on certain symmetric subposets of a product of chains. In Corollary \ref{cor:tabhom} we also use our main result to show a new homomesy result on increasing tableaux of shape $a \times b$ with entries at most $a+b+1$ under $K$-promotion with statistic box entry summation. In Theorem \ref{thm:moregenrecomb}, we generalize the recombination result of Theorem \ref{thm:genrecomb} from a product of chains to any ranked poset. We use this for Corollary \ref{cor:typebmin}, a homomesy result on order ideals of a type B minuscule poset cross a chain of size two under promotion with cardinality statistic. Lastly, Theorem \ref{thm:conjtoggle} explicitly states a bijection between order ideals of a ranked poset under different $n$-dimensional promotions by presenting a conjugating toggle group element.

In Section 2, we begin with introductory definitions and results, much of which is from Striker and Williams \cite{SW2012} and Dilks, Pechenik, and Striker \cite{DPS2017}. In Section 3, we state relevant material from Propp and Roby \cite{PR2015} and Einstein and Propp \cite{EP2014} and work to generalize some of these concepts.
In Section 4, we present our two main results, the homomesy result of Theorem \ref{thm:mainhom} and the generalization of recombination in Theorem \ref{thm:genrecomb}. In Section 5, we present several corollaries, summarized above. In Section 6, we generalize recombination to any ranked poset, obtaining a corollary involving the type B minuscule poset, and, finally, give a theorem presenting a toggle group element to conjugate between different $n$-dimensional promotions.

\section{Rowmotion and promotion background}
We begin by recalling definitions regarding posets, rowmotion, and promotion. 
\begin{definition}
	A \textit{poset} $P$ is a set with a binary relation, denoted $\le$, that is reflexive, weakly antisymmetric, and transitive. Given $e,f \in P$, $f$ \emph{covers} $e$ if $e<f$ and there is no element $x\in P$ such that $e<x<f$. A subset $I$ of $P$ is called an \textit{order ideal} if for any $t \in I$ and $s \le t$ in $P$, $s \in I$. Let $J(P)$ denote the set of order ideals of $P$. A subset $F$ of $P$ is called an \emph{order filter} if for any $t \in F$ and $s \ge t$ in $P$, $s \in F$.
\end{definition}

\begin{definition}
	\label{def:chainposet}
	Let $n \in \mathbb{N}$ and let $[n]$ denote the poset $\{1,2,\dots,n\}$ with the usual less than or equal to $\le$. This is the chain with $n$ elements.
\end{definition}

Some definitions that follow are valid for infinite posets; however, for the rest of this paper, we only consider finite posets. We continue by defining a toggle action on an order ideal of poset.

\begin{definition}
	Let $P$ be a poset. For any $e \in P$, the \textit{toggle} $t_e: J(P) \rightarrow J(P)$ is defined as follows:
	\[ 
	t_e(I)=
	\begin{cases} 
	I \bigcup \{ e \} \qquad & \text{if } e \notin I \text{ and } I \bigcup \{ e \} \in J(P) \\
	I \setminus \{ e \} &\text{if } e \in I \text{ and } I \setminus \{ e \} \in J(P) \\
	I &\text{otherwise.}
	\end{cases}
	\]
\end{definition}

\begin{remark}
	\label{rmk:commute}
	The toggles $t_e$ and $t_f$ commute whenever neither $e$ nor $f$ covers the other.
\end{remark}

\emph{Rowmotion}, denoted Row, is defined as follows.

\begin{definition}
	Let $P$ be a poset and $I \in J(P)$.  $\row(I)$ is the order ideal generated by the minimal elements of $P$ not in $I$. In other words, if $t$ is a minimal element of $P \setminus I$ and $s \le t$, then $s \in \row(I)$.
\end{definition}

However, this is not the only way to view rowmotion.  Cameron and Fon-der-Flaass proved that we may instead toggle elements from top to bottom.

\begin{definition}
	A \emph{linear extension} of a poset $P$ is a bijective function $\linex:P\rightarrow [n]$ where $|P|=n$ such that if $p_1< p_2$ in $P$ then $\linex(p_1)< \linex(p_2)$. Let $\mathcal{L}(P)$ denote the set of linear extensions of $P$.
\end{definition}

\begin{theorem}[\protect{\cite[Lemma 1]{CF1995}}]
	Let $\linex : P \rightarrow [n]$ be in $\mathcal{L}(P)$.  Then $t_{\linex^{-1}(1)}t_{\linex^{-1}(2)} \cdots t_{\linex^{-1}(n)}$ acts as rowmotion.
\end{theorem}

The benefit of the toggle perspective is that we can study other actions that are closely related to rowmotion. In \cite{SW2012}, Striker and Williams defined another action, which they called \emph{promotion}, on order ideals of ranked posets using a projection to a two-dimensional lattice. By defining columns on ranked posets, promotion is the action that toggles columns from left to right. A precise definition of this is stated using Definition \ref{def:hypetoggle} and Proposition \ref{prop:rowandpro}. Note that this promotion action is distinct but related to Sch{\"u}tzenberger's promotion action on linear extensions of posets, defined in \cite{Sch1972}. If we denote promotion on order ideals as Pro, we may see that Row and Pro are linked in the following way.

\begin{theorem}[\protect{\cite[Theorem 5.2]{SW2012}}]
	\label{thm:proroweq}
	For any ranked, finite poset $P$, there is an equivariant bijection between $J(P)$ under $\mathrm{Pro}$ and $J(P)$ under $\mathrm{Row}$.
\end{theorem}

Additionally, in Theorem 5.4 of \cite{SW2012}, Striker and Williams explicitly gave a conjugating toggle element for this bijection. We will generalize this conjugating toggle element result in Theorem \ref{thm:conjtoggle}.

Striker and Williams found that in many cases, it was easier to determine the orbit sizes of Pro compared to Row. More specifically, for some classes of posets, the action of Pro on $J(P)$ is in equivariant bijection with a more easily understood rotation on another object.  As a result, in order to study the orbits of Row, it is often useful to study Pro and apply Theorem \ref{thm:proroweq}.

Dilks, Pechenik, and Striker further generalized the notion of promotion to higher dimensions. Rather than restricting to a two-dimensional lattice projection, they defined promotion for ranked posets with respect to an \emph{$n$-dimensional lattice projection} as toggling by sweeping through the poset with an affine hyperplane in a particular direction \cite{DPS2017}. We postpone the use of lattice projections until Section 6, choosing to present our main results using the natural embedding of the product of $n$ chains into $\mathbb{N}^n$. More specifically, we use the following definition.

\begin{definition}[\protect{\cite[Definition 3.14]{DPS2017}}]
	\label{def:hypetoggle}
	Let $P=[a_1] \times \dots \times [a_n]$ be the product of $n$ chains poset where we consider the elements in the standard $n$-dimensional embedding as vectors in $\mathbb{Z}_{> 0}^n$, and let $v \in \{\pm 1\}^n$.  Let $T_{v}^i$ be the product of toggles $t_x$ for all elements $x$ of $P$ that lie on the affine hyperplane $\langle x,v \rangle=i$.  If there is no such $x$, then this is the empty product, considered to be the identity.  Define \textit{promotion with respect to $v$} as the toggle product Pro$_{v}=\dots T_{v}^{-2} T_{v}^{-1} T_{v}^{0} T_{v}^{1} T_{v}^{2}\dots$
\end{definition}


By Remark \ref{rmk:commute}, toggles commute if there is no covering relation between their corresponding poset elements. This guarantees that $\mathrm{Pro}_v$ is well-defined.

\begin{remark}[\protect{\cite[Lemma 3.16]{DPS2017}}]
	Two elements of the poset that lie on the same affine hyperplane $\langle x,v \rangle=i$ cannot be part of a covering relation.
\end{remark}

Now that we established $\mathrm{Pro}_{v}$ and verified it is well-defined, we can relate it to the previously established Row.

\begin{proposition}[\protect{\cite[Remark 3.17, Proposition 3.18]{DPS2017}}]
	\label{prop:rowandpro}
	For a finite ranked poset $P$, $\mathrm{Pro}_{(1,1,\dots,1)}=\mathrm{Row}$. Additionally, $\mathrm{Pro}_{(-1,1)}$ is the two-dimensional promotion action $\mathrm{Pro}$.
\end{proposition}

The orbit structure of order ideals of certain posets under rowmotion and promotion has been well-studied. Another phenomenon, isolated by Propp and Roby, appears frequently among many of these posets and will be the subject of the next section. 

\section{The homomesy phenomenon and recombination}
In this section, we define homomesy and state known results in two dimensions. We will generalize these results to higher dimensions in Section 4 and to more general posets in Section 6.

\begin{definition}
	\label{def:homomesy}
	Given a finite set $S$, an action $\tau:S \rightarrow S$, and a statistic $f:S \rightarrow k$ where $k$ is a field of characteristic zero, we say that $(S, \tau, f)$ exhibits \textit{homomesy} if there exists $c \in k$ such that for every $\tau$-orbit $\orb$ 
	\begin{center}
		$\displaystyle\frac{1}{\# \orb} \sum_{x \in \orb} f(x) = c$
	\end{center}
	where $\# \orb$ denotes the number of elements in $\orb$. If such a $c$ exists, we will say the triple is \emph{$c$-mesic}.
\end{definition}

Homomesy results have been observed in many well-known combinatorial objects. To expound on one of these examples, Propp and Roby proved the following results on a product of chains.

\begin{theorem}[\protect{\cite[Theorem 19]{PR2015}}]
	\label{thm:axbPro}
	Let $f$ be the cardinality statistic. Then $(J([a] \times [b]), \mathrm{Pro}, f)$ is $ab/2$-mesic.
\end{theorem}

\begin{theorem}[\protect{\cite[Theorem 23]{PR2015}}]
	\label{thm:axbRow}
	Let $f$ be the cardinality statistic. Then $(J([a] \times [b]), \mathrm{Row}, f)$ is $ab/2$-mesic.
\end{theorem}
It is beneficial to study $J([a] \times [b])$ under Pro rather than Row, as $J([a] \times [b])$ under Pro is in equivariant bijection with an object that rotates. This fact makes the proof of Theorem \ref{thm:axbPro} fairly straightfoward. Propp and Roby also have a direct proof of Theorem \ref{thm:axbRow} in \cite{PR2015}; however, it is much more technical than in the promotion case.  Einstein and Propp found a more elegant way to prove Theorem \ref{thm:axbRow} in \cite{EP2014}, with further details in \cite{EP2013}, by using a technique they called \emph{recombination}. Their recombination technique gives an equivariant bijection between $J([a] \times [b])$ under Row and Pro. From this, we may start with an orbit from $J([a] \times [b])$ under Row and take sequential layers from order ideals to form a new orbit under Pro. We first introduce some useful notation.

\begin{definition}
	\label{def:vhat}
	Suppose $v=(v_1,v_2,\dots, v_n)\in \mathbb{Z}^n$. Given $\gamma \in [n]$, let \\$v^{\widehat{\gamma}}=(v_1,v_2,\dots,v_{\gamma -1}, v_{\gamma +1}, \dots, v_n)$.
\end{definition}

We define our layers in the following way.

\begin{definition}
	\label{def:layers}
	Fix $\gamma \in [n]$ and let $P=[a_1] \times \dots \times [a_n]$. Define the $j$th $\gamma$-layer of $I\in J(P)$ as \[L_{\gamma}^j(I)=\{(i_1,i_2,\ldots, i_n)\in I \ | \ i_{\gamma}=j\}.\] We will denote
	\[L_{\gamma}^j \coloneqq L_{\gamma}^j(P).\]
	Additionally, given $L_{\gamma}^j$ and $L_{\gamma}^j(I)$, define \[L_{\gamma}^j(I)^{\widehat{\gamma}}=\{(i_1,i_2,\ldots, i_n)^{\widehat{\gamma}} \ | \ (i_1,i_2,\ldots, i_n)\in L_{\gamma}^j(I)\}\] and
	\[(L_{\gamma}^{j})^{\widehat{\gamma}}\coloneqq L_{\gamma}^j(P)^{\widehat{\gamma}}\]
\end{definition}

$\gamma$ tells us the direction of our layers while $j$ signifies which of the layers we are taking in that direction.

When $n=2$, Einstein and Propp referred to each $L_{1}^j$ as a negative fiber of $P$; we use the notation $L_{\gamma}^j$ and $L_{\gamma}^j(I)$ as it more naturally describes our layers when we generalize to higher dimensions. Furthermore, we define $(L_{\gamma}^j)^{\widehat{\gamma}}$ and $L_{\gamma}^j(I)^{\widehat{\gamma}}$, which removes the $j$th coordinate, as it will be useful to view our layers in the $(n-1)$-dimensional setting.

Using the idea of layers, Einstein and Propp defined the concept of recombination and proved the following proposition, which we restate in the above notation. See Figure \ref{fig:3x3recomb} for an example. 

\begin{definition}
	\label{def:2drecomb}
	Let $I \in J([a] \times [b])$. Define the \emph{recombination} of $I$ as $\Delta I=\bigcupdot_{j}L_{1}^j(\row^{j-1}(I))$.
\end{definition}

\begin{proposition}[\protect{\cite[Theorem 12]{EP2013}}]
	\label{prop:2drecomb}
	Let $I \in J([a] \times [b])$. Then $\mathrm{Pro}(\Delta I)=\bigcupdot_{j}L_{1}^j(\row^{j}(I))=\Delta(\row(I))$.
\end{proposition}

\begin{figure}[htbp]
	\centering
	\begin{subfigure}{.45\textwidth}
		\centering
		\includegraphics[width=1\linewidth]{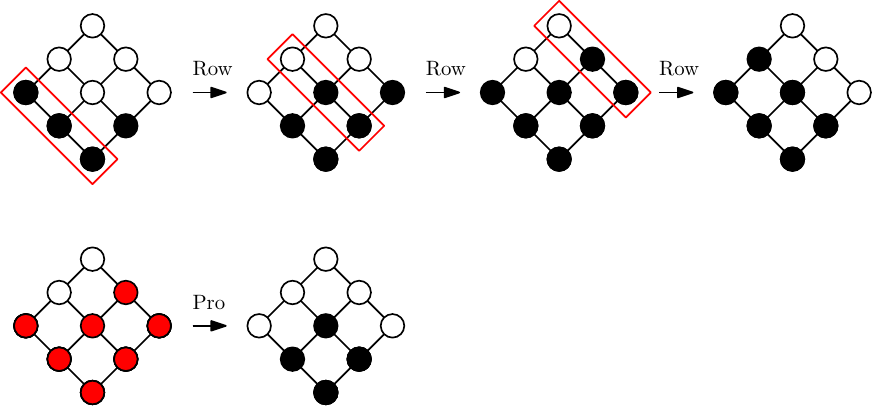}
		\caption{From an orbit of Row, we use $L_{1}^1(I), L_{1}^2(\mathrm{Row}(I))$, and $L_{1}^3(\mathrm{Row}^{2}(I))$ to form a new order ideal, denoted here in red.}
		\label{fig:3x3recomb1}
	\end{subfigure}\hfill
	\begin{subfigure}{.45\textwidth}
		\centering
		\includegraphics[width=1\linewidth]{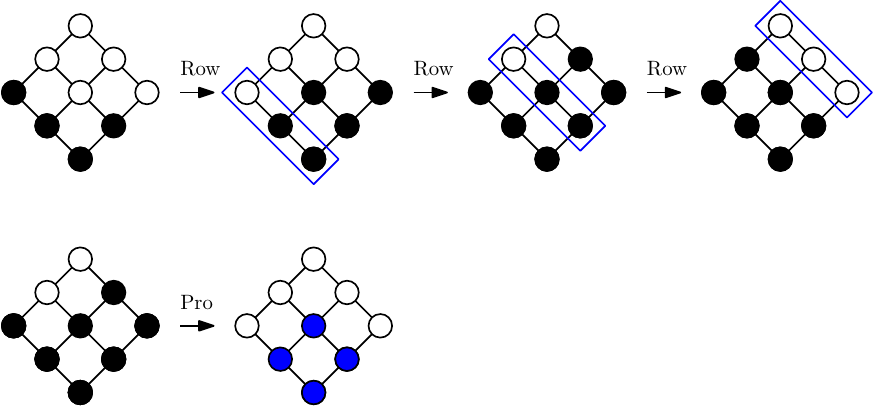}
		\caption{From the same orbit of Row, we use $L_{1}^1(\mathrm{Row}(I)), L_{1}^2(\mathrm{Row}^2(I))$, and $L_{1}^3(\mathrm{Row}^3(I))$ to form a new order ideal, denoted here in blue.}
		\label{fig:3x3recomb2}
	\end{subfigure}
	\caption{Performing Pro on the red order ideal results in the blue order ideal.}
	\label{fig:3x3recomb}
\end{figure}

The idea behind recombination is the following: we take a single layer from each order ideal in a sequence of order ideals from a rowmotion orbit to form the layers of a new order ideal. Proposition \ref{prop:2drecomb} tells us that if we apply promotion to this new order ideal, the result is the same as if we move one step forward in the rowmotion orbit and apply recombination again. In other words, recombination gives an equivariant bijection between $J([a] \times [b])$ under Pro and $J([a] \times [b])$ under Row.

In Theorem \ref{thm:genrecomb}, we generalize this notion to higher dimensional products of chains. Before doing so, however, we observe an important property of Row and Pro and how their toggles commute in the $[a] \times [b]$ case. To state this we introduce an additional definition, which will also prove useful for discussing commuting toggles in $n$-dimensions.

\begin{definition}
	Let $P=[a_1] \times \dots \times [a_n]$, $v \in \{\pm 1\}^n$, and $\gamma \in [n]$. Define $T^{j}_{\mathrm{Pro}_{v^{\widehat{\gamma}}}}$ as the toggle product of Pro$_{v^{\widehat{\gamma}}}$ on $(L_{\gamma}^j)^{\widehat{\gamma}}$.
\end{definition}

With this notation, given an $n$-dimensional vector $v$, we define a product of toggles on an $(n-1)$-dimensional product of chains with the order of toggles given by Pro$_{v^{\widehat{\gamma}}}$. The following proposition shows that for $[a] \times [b]$, we can express Row and Pro using these toggle products. In Theorem \ref{thm:ndcommute}, we will show this holds more generally for $[a_1] \times \dots \times [a_n]$ and $\mathrm{Pro}_v$.

\begin{proposition}[\protect{\cite[Section 8]{EP2013}} \protect{\cite[Theorem 5.4]{SW2012}}]
	\label{prop:2dcommute}
	Let \new{$P=[a] \times [b]$}. $\row=\mathrm{Pro_{(1,1)}}=\prod_{j=1}^{a} T^{j}_{\mathrm{Pro}_{(1,1)^{\widehat{1}}}}$ and $\mathrm{Pro}=\prod_{j=1}^{a} T^{a+1-j}_{\mathrm{Pro}_{(-1,1)^{\widehat{1}}}}$.
\end{proposition}

In other words, we can commute the toggles of Row so we toggle $L_{1}^{a}$, followed by $L_{1}^{a-1}$, and so on, toggling each layer from top to bottom. For example, in Figure \ref{fig:commute1}, we can commute the red toggle with both blue toggles, as the red element does not have a covering relation with either blue element. Therefore, when performing Row we can toggle both blue elements before the red element, and hence all of $L_{1}^{3}$ before the red element. Similar reasoning applies for each $L_{1}^{j}$, and as a result we can perform Row by toggling in the order denoted in Figure \ref{fig:commute2}, where layer 3 is first, layer 2 is second, and layer 1 third. Additionally, the toggle order in each layer is denoted with an arrow. In other words, $\row=T^{1}_{\mathrm{Pro}_{(1,1)^{\widehat{1}}}}T^{2}_{\mathrm{Pro}_{(1,1)^{\widehat{1}}}}T^{3}_{\mathrm{Pro}_{(1,1)^{\widehat{1}}}}$. Note that Pro is similar, except we would toggle layer 1 first, then layer 2: $\mathrm{Pro}=T^{3}_{\mathrm{Pro}_{(-1,1)^{\widehat{1}}}}T^{2}_{\mathrm{Pro}_{(-1,1)^{\widehat{1}}}}T^{1}_{\mathrm{Pro}_{(-1,1)^{\widehat{1}}}}$. The toggle order of each layer is identical for both as $\mathrm{Pro}_{(1,1)^{\widehat{1}}}=\mathrm{Pro}_{(-1,1)^{\widehat{1}}}=\mathrm{Pro}_{(1)}$.

\begin{figure}[htbp]
	\centering
	\begin{subfigure}{.45\textwidth}
		\centering
		\includegraphics[width=.4\linewidth]{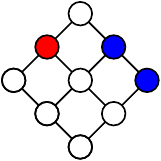}
		\caption{We can commute the toggle of either blue element with the red element, as there is no covering relation between them.}
		\label{fig:commute1}
	\end{subfigure}\hfill
	\begin{subfigure}{.45\textwidth}
		\centering
		\includegraphics[width=.5\linewidth]{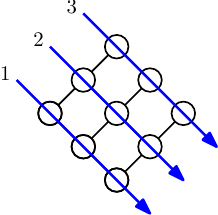}
		\caption{We toggle layer 3, then layer 2, then layer 1, with arrows denoting toggle order in each layer. This toggle order is equivalent to Row by commuting toggles.}
		\label{fig:commute2}
	\end{subfigure}
	\caption{}
	\label{fig:commute}
\end{figure}

\section{Homomesy on $J([2] \times [b] \times [c])$ and higher dimensional recombination}
Having explored known homomesy results in Section 3, we state our first main result, a homomesy result on order ideals of $[2] \times [b] \times [c]$ under promotion with cardinality statistic (Theorem \ref{thm:mainhom}). This is a generalization of two results of Propp and Roby: order ideals of $[a] \times [b]$ under promotion and rowmotion with cardinality statistic exhibit homomesy (Theorems \ref{thm:axbPro} and \ref{thm:axbRow}). Additionally, we use symmetry to show homomesy results on order ideals of $[a] \times [2] \times [c]$ and $[a] \times [b] \times [2]$ under promotion with cardinality statistic (Corollary \ref{cor:maincor1}). We also generalize the definition of recombination on a product of two chains (Definition \ref{def:2drecomb}) to a product of $n$ chains (Definition \ref{def:ndimrecomb}). Our second main result generalizes the connection between rowmotion and promotion under recombination from a product of two chains (Proposition \ref{prop:2drecomb}) to a product of $n$ chains (Theorem \ref{thm:genrecomb}). We conclude this section showing that order ideals of arbitrary products of three chains under promotion with cardinality statistic do not exhibit homomesy (Proposition \ref{prop:334}), nor do order ideals of an arbitrary product of $n$ two element chains (\ref{prop:22222}). These results show that our main homomesy result (Theorem \ref{thm:mainhom}) does not generalize further. In the next section, we discuss a partial generalization.

\begin{theorem}
	\label{thm:mainhom}
	Let $f$ be the cardinality statistic and $v \in \{\pm 1\}^n$.  The triple $(\prodchainstwo, \mathrm{Pro}_v, f)$ is $bc$-mesic.
\end{theorem}

In order to prove Theorem \ref{thm:mainhom}, we will define the notion of recombination for a product of chains in full generality.

\begin{definition}
	\label{def:ndimrecomb}
	Let $P=[a_1] \times \dots \times [a_n]$, $v \in \{\pm 1\}^n$, and $I \in J(P)$. Define $\recomb{v}{\gamma}=\bigcupdot_{j}L_{\gamma}^j(\mathrm{Pro}_v^{j-1}(I))$ where $\gamma \in [n]$. We will call $\recomb{v}{\gamma}$ the \emph{$(v,\gamma)-$recombination} of $I$. When context is clear, we will suppress the $(v,\gamma)$.
\end{definition}

The idea behind \prorecomb{v}{\gamma} is the same as in the two-dimensional case: we take sequentially one layer from each order ideal from a promotion orbit to form the layers of a new order ideal. See Figure \ref{fig:2x3x2recomb} for an example. In addition to generalizing recombination to $n$ dimensions, we also generalize Proposition \ref{prop:2dcommute} to $n$ dimensions.

\begin{theorem}
	\label{thm:ndcommute}
	Let \new{$P=[a_1] \times \dots \times [a_n]$}, $v \in \{\pm 1\}^n$, and $\gamma \in [n]$. Then $\mathrm{Pro}_{v}=\prod_{j=1}^{a_\gamma} T^{\alpha}_{\mathrm{Pro}_{v^{\widehat{\gamma}}}}$ where
	\[
	\alpha =
	\begin{cases} 
	j   \hfill & \text{if } v_\gamma=1 \\
	a_\gamma+1-j \hfill & \text{if } v_\gamma=-1. \\
	\end{cases}
	\]
\end{theorem}
\begin{proof}
	Suppose $x:=(x_1,\dots,x_n), y:=(y_1,\dots,y_n) \in \new{P}$ with $x \in \new{L_{\gamma}^j}$ and $y \in \new{L_{\gamma}^k}$ for some $j$ and $k$. We want to show that $x$ and $y$ are toggled in the same order in Pro$_{v}$ and $\prod_{j=1}^{a_\gamma} T^{\alpha}_{\mathrm{Pro}_{v^{\widehat{\gamma}}}}$.
	
	\textbf{Case $j\ne k$:} Without loss of generality, $j>k$.  Furthermore, we can assume $x_\gamma = y_\gamma+1$ and $x_i=y_i$ for $i \ne \gamma$. If this were not the case, $x$ and $y$ could not have a covering relation and we could commute the toggles.
	
	If $v_\gamma=1:$ In $\prod_{j=1}^{a_\gamma} T^{\alpha}_{\mathrm{Pro}_{v^{\widehat{\gamma}}}}$, $x$ is toggled before $y$ by definition. Additionally, \[\langle x,v \rangle = v_1 x_1 + \dots + v_\gamma x_\gamma + \dots + v_n x_n > v_1 y_1 + \dots + v_\gamma y_\gamma + \dots + v_n y_n = \langle y,v \rangle\] and so $x$ is toggled before $y$ in Pro$_{v}$.
	
	If $v_\gamma=-1:$ In $\prod_{j=1}^{a_\gamma} T^{\alpha}_{\mathrm{Pro}_{v^{\widehat{\gamma}}}}$, $y$ is toggled before $x$ by definition. Additionally, \[\langle x,v \rangle = v_1 x_1 + \dots + v_\gamma x_\gamma + \dots + v_n x_n < v_1 y_1 + \dots + v_\gamma y_\gamma + \dots + v_n y_n = \langle y,v \rangle\] and so $y$ is toggled before $x$ in Pro$_{v}$.
	
	\textbf{Case $j=k$:} In other words, $x_\gamma=y_\gamma$. Therefore,
	\begin{align*}
	\langle x,v \rangle > \langle y,v \rangle  \iff & v_1 x_1 + \dots + v_\gamma x_\gamma + \dots + v_n x_n > v_1 y_1 + \dots + v_\gamma y_\gamma + \dots + v_n y_n \\
	\iff & v_1 x_1 + \dots + v_{\gamma-1} x_{\gamma-1} + v_{\gamma+1} x_{\gamma+1} + \dots + v_n x_n > \\ & v_1 y_1 + \dots + v_{\gamma-1} y_{\gamma-1} + v_{\gamma+1} y_{\gamma+1} + \dots + v_n y_n \\
	\iff & \langle x^{\widehat{\gamma}},v^{\widehat{\gamma}} \rangle > \langle y^{\widehat{\gamma}},v^{\widehat{\gamma}} \rangle
	\end{align*}
	Therefore, $x$ can be toggled before $y$ in Pro$_v$ if and only if $x$ can be toggled before $y$ in $\prod_{j=1}^{a_\gamma} T^{\alpha}_{\mathrm{Pro}_{v^{\widehat{\gamma}}}}$.
\end{proof}

In other words, if we want to apply Pro$_v$, we can commute our toggles to toggle by layers of the form $L_{\gamma}^j$ instead of using the toggle order given in Definition \ref{def:hypetoggle}. More specifically, if $v_\gamma=1$, we toggle in the order of $L_{\gamma}^{a_\gamma}, L_{\gamma}^{a_\gamma-1},\dots,L_{\gamma}^{1}$. If $v_\gamma=-1$, we toggle in the order of $L_{\gamma}^{1}, L_{\gamma}^{2},\dots,L_{\gamma}^{a_\gamma}$. 

Now that we have established $n$-dimensional recombination and toggle commutation, we determine conditions under which \prorecomb{v}{\gamma} results in an order ideal.

\begin{figure}[htbp]
	\centering
	\begin{subfigure}{.45\textwidth}
		\centering
		\includegraphics[width=1\linewidth]{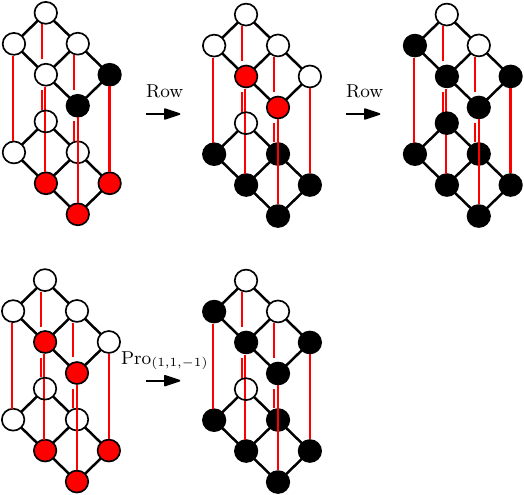}
		\caption{From an orbit of Row, we use $L_{3}^1(I)$ and $L_{3}^2(\mathrm{Row}(I))$ to form a new order ideal, denoted here in red.}
		\label{fig:2x3x2recomb1}
	\end{subfigure}\hfill
	\begin{subfigure}{.45\textwidth}
		\centering
		\includegraphics[width=1\linewidth]{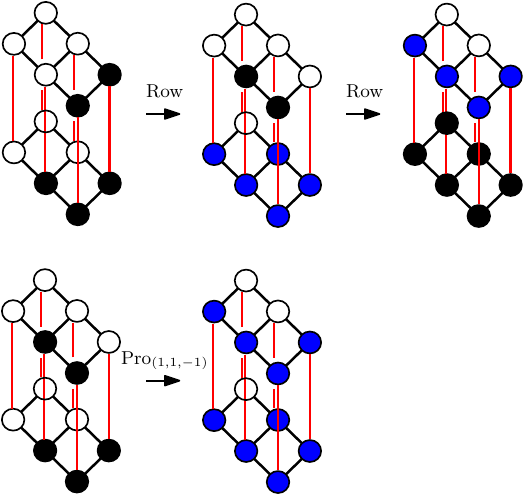}
		\caption{From the same orbit of Row, we use $L_{3}^1(\mathrm{Row}(I))$ and $L_{3}^2(\mathrm{Row}^2(I))$ to form a new order ideal, denoted here in blue.}
		\label{fig:2x3x2recomb2}
	\end{subfigure}
	\caption{Performing Pro$_{(1,1,-1)}$ on the red order ideal results in the blue order ideal.}
	\label{fig:2x3x2recomb}
\end{figure}

\begin{lemma}
	\label{lemma:orderideal}
	Let $I\in J([a_1] \times \dots \times [a_n])$. Suppose we have $v \in \{\pm 1\}^n$ and $\gamma$ such that $v_\gamma=1$. Then $\recomb{v}{\gamma}$ is an order ideal of P.
\end{lemma}
\begin{proof}
	Suppose $(i_1,\dots,i_n) \in \recomb{v}{\gamma}$. By definition, $(i_1,\dots,i_j-1,\dots,i_n) \in \recomb{v}{\gamma}$ for $j \ne \gamma$. To show that $\recomb{v}{\gamma}$ is an order ideal, it suffices to show $(i_1,\dots,i_\gamma-1,\dots,i_n) \in \recomb{v}{\gamma}$ for $i_\gamma \ge 2$; if $i_\gamma=1$ there is nothing to show.  Because $(i_1,\dots,i_n) \in \recomb{v}{\gamma}$, we have $(i_1,\dots,i_n) \in L_{\gamma}^{i_\gamma}(\mathrm{Pro}_v^{i_{\gamma}-1}(I))$. By Theorem \ref{thm:ndcommute}, Pro$_{v}=\prod_{j=1}^{a_\gamma} T^{\new{a_\gamma+1-j}}_{\mathrm{Pro}_{v^{\widehat{\gamma}}}}$, which implies we can commute the toggle relations in Pro$_{v}$ so that $L_{\gamma}^{i_\gamma}$ is toggled before $L_{\gamma}^{i_{\gamma-1}}$. As a result, we must have $(i_1,\dots,i_\gamma-1,\dots,i_n) \in L_{\gamma}^{i_\gamma-1}(\mathrm{Pro}_v^{i_{\gamma}-2}(I))$.  Therefore, $(i_1,\dots,i_\gamma-1,\dots,i_n) \in \recomb{v}{\gamma}$.
\end{proof}

We can now state our second main result, which shows when recombination gives us an equivariant bijection from $J([a_1] \times \dots \times [a_n])$ under $\mathrm{Pro}_{u}$ to $J([a_1] \times \dots \times [a_n])$ under $\mathrm{Pro}_{v}$. This result will allow us to prove Theorem~\ref{thm:mainhom}.
\begin{theorem}
	\label{thm:genrecomb}
	Let $I\in J([a_1] \times \dots \times [a_n])$. Suppose we have $u,v \in \{\pm 1\}^n$ and $\gamma$ such that $v_\gamma=1$, $u_\gamma =-1$, and ${v^{\widehat{\gamma}}} = {u^{\widehat{\gamma}}}$. Then $\mathrm{Pro}_{u}(\recomb{v}{\gamma})=\Delta_v^{\gamma}(\mathrm{Pro}_{v}(I))$.
\end{theorem}
\begin{proof}
	First, note that $\recomb{v}{\gamma}$ is an order ideal by Lemma \ref{lemma:orderideal}. Also note that Pro$_{v}=\prod_{j=1}^{a_\gamma} T^{j}_{\mathrm{Pro}_{v^{\widehat{\gamma}}}}$ and Pro$_{u}=\prod_{j=1}^{a_\gamma} T^{a_\gamma+1-j}_{\mathrm{Pro}_{u^{\widehat{\gamma}}}}$ by Theorem \ref{thm:ndcommute}. We will show $\mathrm{Pro}_{u}(\recomb{v}{\gamma}$)=$\Delta_v^{\gamma}(\mathrm{Pro}_{v}(I))$ by showing $L_{\gamma}^{k}(\mathrm{Pro}_{u}(\recomb{v}{\gamma}))=L_{\gamma}^{k}(\Delta_v^{\gamma}(\mathrm{Pro}_{v}(I)))$ for each $k \in \{1,2,\dots,a_\gamma\}$. There are three cases.
	
	\textbf{Case $1 < k < a_\gamma$:} Let $J=\mathrm{Pro}_v^{k-1}(I)$. We can commute the toggles of Pro$_{v}$ so that $L_{\gamma}^{k+1}$ of $J$ is toggled before $L_{\gamma}^{k}$ of $J$, which is toggled before $L_{\gamma}^{k-1}$ of $J$. Thus, when applying the toggles of Pro$_{v}$ to $L_{\gamma}^{k}$ of $J$, the layer above is $L_{\gamma}^{k+1}(\mathrm{Pro}_v(J))$ whereas the layer below is $L_{\gamma}^{k-1}(J)$. Additionally, we can also commute the toggles of Pro$_{u}$ so $L_{\gamma}^{k-1}$ of $\recomb{v}{\gamma}$ is toggled before $L_{\gamma}^{k}$ of $\recomb{v}{\gamma}$, which is toggled before $L_{\gamma}^{k+1}$ of $\recomb{v}{\gamma}$. Therefore, when applying the toggles of Pro$_{u}$ to $L_{\gamma}^{k}$ of $\recomb{v}{\gamma}$, the layer below is $L_{\gamma}^{k-1}(\mathrm{Pro}_u(\recomb{v}{\gamma}))$, whereas the layer above is $L_{\gamma}^{k+1}(\recomb{v}{\gamma})$. However, $L_{\gamma}^{k-1}(\mathrm{Pro}_u(\recomb{v}{\gamma}))=L_{\gamma}^{k-1}(J)$, $L_{\gamma}^{k}(\recomb{v}{\gamma})=L_{\gamma}^{k}(J)$, and $L_{\gamma}^{k+1}(\recomb{v}{\gamma})=L_{\gamma}^{k+1}(\mathrm{Pro}_v(J))$. Therefore, when applying Pro$_{v}$ to $L_{\gamma}^{k}$ of $J$ and Pro$_{u}$ to $L_{\gamma}^{k}$ of $\recomb{v}{\gamma}$, both layers are the same and have the same layers above and below them. Because $u^{\widehat{\gamma}}=v^{\widehat{\gamma}}$, we have $\mathrm{Pro}_{u^{\widehat{\gamma}}}=\mathrm{Pro}_{v^{\widehat{\gamma}}}$ and so the result of toggling this layer is $L_{\gamma}^{k}($Pro$_{u}(\recomb{v}{\gamma}))$, which is the same as $L_{\gamma}^{k}($Pro$_{v}(J))=L_{\gamma}^{k}(\mathrm{Pro}_v^{k}(I))=L_{\gamma}^{k}(\Delta_v^{\gamma}(\mathrm{Pro}_{v}(I)))$.
	
	\textbf{Case $k = 1$:} As above, when applying Pro$_{v}$ to $L_{\gamma}^{1}$ of $I$ and Pro$_{u}$ to $L_{\gamma}^{1}$ of $\recomb{v}{\gamma}$, both of these layers are the same, along with the layers above them. Because $k=1$, there is not a layer below. As above, $\mathrm{Pro}_{u^{\widehat{\gamma}}}=\mathrm{Pro}_{v^{\widehat{\gamma}}}$ and so we again obtain $L_{\gamma}^{1}(\mathrm{Pro}_{u}(\recomb{v}{\gamma}))=L_{\gamma}^{1}(\Delta_v^{\gamma}(\mathrm{Pro}_{v}(I)))$.
	
	\textbf{Case  $k=a_\gamma$:} Again, as above, when applying Pro$_{v}$ to $L_{\gamma}^{a_\gamma}$ of $\new{J}$ and Pro$_{u}$ to $L_{\gamma}^{a_\gamma}$ of $\new{\recomb{v}{\gamma}}$, both of these layers are the same along with the layers below them. Because $k=a_\gamma$ there is not a layer above. Again, $\mathrm{Pro}_{u^{\widehat{\gamma}}}=\mathrm{Pro}_{v^{\widehat{\gamma}}}$ and so $L_{\gamma}^{a_\gamma}(\mathrm{Pro}_{u}(\recomb{v}{\gamma}))=L_{\gamma}^{a_\gamma}(\Delta_v^{\gamma}(\mathrm{Pro}_{v}(I)))$.
\end{proof}

\begin{figure}[htbp]
	\centering
	\includegraphics[scale=1]{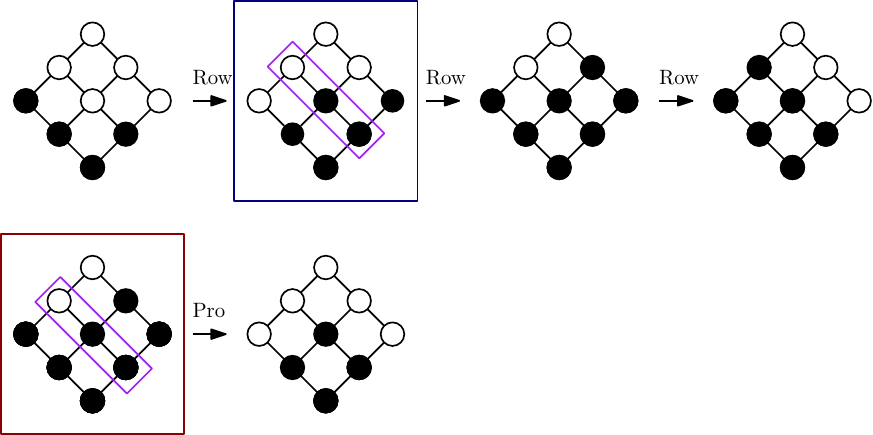}
	\caption{We refer to the same example as in Figure \ref{fig:3x3recomb}. The boxed purple layers correspond under recombination. In Example \ref{ex:recomb}, we demonstrate the idea of the proof using the order ideals in the large blue and red boxes.}
	\label{fig:proofpicture1}
\end{figure}

\begin{figure}[htbp]
	\centering
	\includegraphics[width=.7\linewidth]{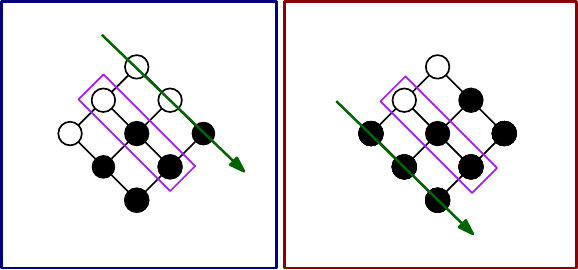}
	\caption{When performing Row to the left figure, $L_{1}^3(I)$ is toggled first in the direction indicated. When performing Pro to the right figure, $L_{1}^1(I)$ is toggled first in the direction indicated.}
	\label{fig:proofpicture2}
\end{figure}

\begin{figure}[htbp]
	\centering
	\includegraphics[width=.7\linewidth]{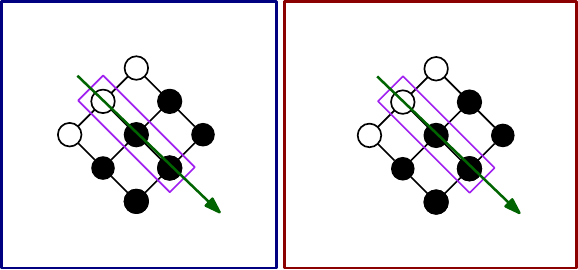}
	\caption{After performing the toggles from Figure \ref{fig:proofpicture2}, the order ideal in the left figure now has $L_{1}^3(I)$ from the order ideal that follows it in the orbit of Row. Similarly, the order ideal in the left figure has $L_{1}^1(I)$ from the order ideal that follows it in the orbit of Pro. When performing toggles on the purple layer, the three layers are the same.}
	\label{fig:proofpicture3}
\end{figure}

\begin{example}
	\label{ex:recomb}
	To see an example of the proof technique, we will refer to Figures \ref{fig:proofpicture1}, \ref{fig:proofpicture2}, and \ref{fig:proofpicture3}. We begin in Figure \ref{fig:proofpicture1} with the same orbit under Row as in Figure \ref{fig:3x3recomb}. Let $I$ denote the first order ideal in this orbit; using recombination we form the order ideal $\recomb{(1,1)}{1}$. We want to verify that by forming sequential recombination order ideals, we obtain an orbit under Pro. We will do so by showing that corresponding layers in the Row orbit and the sequence of recombination order ideals result in the same layer after performing Row and Pro, respectively. The boxed purple layers $L_{1}^2(I)$ in both orbits of Figure~\ref{fig:proofpicture1} correspond under recombination. We can commute the toggles of Row as we did in Figure~\ref{fig:commute2}. We can also commute the toggles of Pro so we toggle layer 3, then layer 2, then layer 1 in Figure~\ref{fig:commute2}. This means when performing Row, we first toggle the layer indicated by the green arrow in the left figure in Figure~\ref{fig:proofpicture2}; similarly for Pro and the right figure in Figure~\ref{fig:proofpicture2}. Then, the next step of both Row and Pro is to toggle the boxed purple layer, as seen in Figure~\ref{fig:proofpicture3}. We see that when we perform this step of Row and Pro, the boxed purple layer, the layer above, and the layer below are the same. Because we are toggling the same direction along the boxed purple layer, we are guaranteed the same result in both cases.
\end{example}

We have three immediate corollaries that will be useful in the proof of Theorem \ref{thm:mainhom}.
\begin{corollary}
	\label{cor:rowrecomb}
	$\mathrm{Pro}_{(1,1,-1)}(\recomb{(1,1,1)}{3})=\Delta_{(1,1,1)}^{3}(\mathrm{Pro}_{(1,1,1)}(I))$.
\end{corollary}
\begin{proof}
	$v=(1,1,1)$, $u=(1,1,-1)$, and $\gamma=3$ satisfy the assumptions of Theorem \ref{thm:genrecomb}.
\end{proof}
\begin{corollary}
	\label{cor:prorecomb}
	$\mathrm{Pro}_{(-1,1,-1)}(\recomb{(1,1,-1)}{1})=\Delta_{(1,1,-1)}^{1}(\mathrm{Pro}_{(1,1,-1)}(I))$.
\end{corollary}
\begin{proof}
	$v=(1,1,-1)$, $u=(-1,1,-1)$, and $\gamma=1$ satisfy the assumptions of Theorem~\ref{thm:genrecomb}.
\end{proof}
\begin{corollary}
	\label{cor:thirdrecomb}
	$\mathrm{Pro}_{(1,-1,-1)}(\recomb{(1,1,-1)}{2})=\Delta_{(1,1,-1)}^{2}(\mathrm{Pro}_{(1,1,-1)}(I))$.
\end{corollary}
\begin{proof}
	$v=(1,1,-1)$, $u=(1,-1,-1)$, and $\gamma=2$ satisfy the assumptions of Theorem~\ref{thm:genrecomb}.
\end{proof}

Recombination gives us an equivariant bijection between order ideals under different promotion actions. From this, we have a connection between orbits of different promotion actions. Suppose $v$ and $u$ are as in Theorem \ref{thm:genrecomb}. If we find the recombination of each order ideal in an orbit of $\mathrm{Pro}_{v}$, we obtain a sequence of order ideals that form an orbit under $\mathrm{Pro}_{u}$.

\begin{remark}
	\label{rmk:recombbij}
	Let $u,v$ be as in Theorem \ref{thm:genrecomb} and let $\orb$ be an orbit of order ideals in $J([a_1] \times \dots \times [a_n])$ under $\mathrm{Pro}_{u}$. There is a unique orbit $\orb'$ under $\mathrm{Pro}_{v}$ where the recombination of $\orb'$ is $\orb$. In other words, if we start with an orbit under $\mathrm{Pro}_{u}$, we can invert recombination to get an orbit under $\mathrm{Pro}_{v}$. More specifically, if we start with an orbit of $\prodchainstwo$ under Pro$_{(-1,1,-1)}$, we can acquire an orbit of $\prodchainstwo$ under Pro$_{(1,1,-1)}$.
\end{remark}

This observation will be used to show $\prodchainstwo$ exhibits homomesy under Pro$_{(-1,1,-1)}$ and Pro$_{(1,-1,-1)}$.

To prove Theorem \ref{thm:mainhom}, we relate the order ideals of our posets to increasing tableaux.  To do so, we first need a map from $\prodchains$ to increasing tableaux defined by Dilks, Pechenik, and Striker.

\begin{definition}
	An \textit{increasing tableau} of shape $\lambda$ is a filling of boxes of partition shape $\lambda$ with positive integers such that the entries strictly increase from left to right across rows and strictly increase from top to bottom along columns. We will use Inc$^{q}(\lambda)$ to indicate the set of increasing tableaux of shape $\lambda$ with entries at most $q$.
\end{definition}

Figure \ref{fig:inctab} shows an increasing tableaux in Inc$^{q}(3,3,1)$ where $q$ can be any integer greater than or equal to 6.

\begin{figure}[htbp]
	\ytableausetup{centertableaux}
	\begin{ytableau}
		1 & 2 & 4 \\
		2 & 4 & 5 \\
		6
	\end{ytableau}
	\caption{An increasing tableaux of shape $\lambda = (3,3,1)$.}
	\label{fig:inctab}
\end{figure}

\begin{definition}{\cite[Section 4.1]{DPS2017}}
	\label{def:prodincbij}
	Define a map $\Psi : \prodchains \rightarrow \inc{a+b+c-1}{a \times b}$ in the following way.  Let $I \in \prodchains$.  We can view $I$ as a pile of cubes in an $a \times b \times c$ box; we then project onto a Young diagram of shape $a \times b$. More specifically, record in position $(i,j)$ the number of boxes of $I$ with coordinate $(i,j,k)$ for some $k \in [c]$.  This results in a filling of a Young diagram of shape $a \times b$ with nonnegative entries that weakly decrease from left to right and top to bottom.  By rotating the diagram 180$\degree$, our Young diagram is now weakly increasing in rows and columns.  Now for the label in position $(i,j)$, increase the label by $i+j-1$.  This results in an increasing tableau, which we denote $\Psi(I)$.
\end{definition}

In Figure \ref{fig:tableaubijectionexample}, we see an example of $\Psi : J([2] \times [3] \times [2]) \rightarrow \inc{6}{2 \times 3}$, as defined in Definition \ref{def:prodincbij}. We also give a definition for $K$-promotion, an action defined on increasing tableaux by Pechenik in \cite{Pechenik2014}.

\begin{figure}[htbp]
	\includegraphics[width=.25\linewidth]{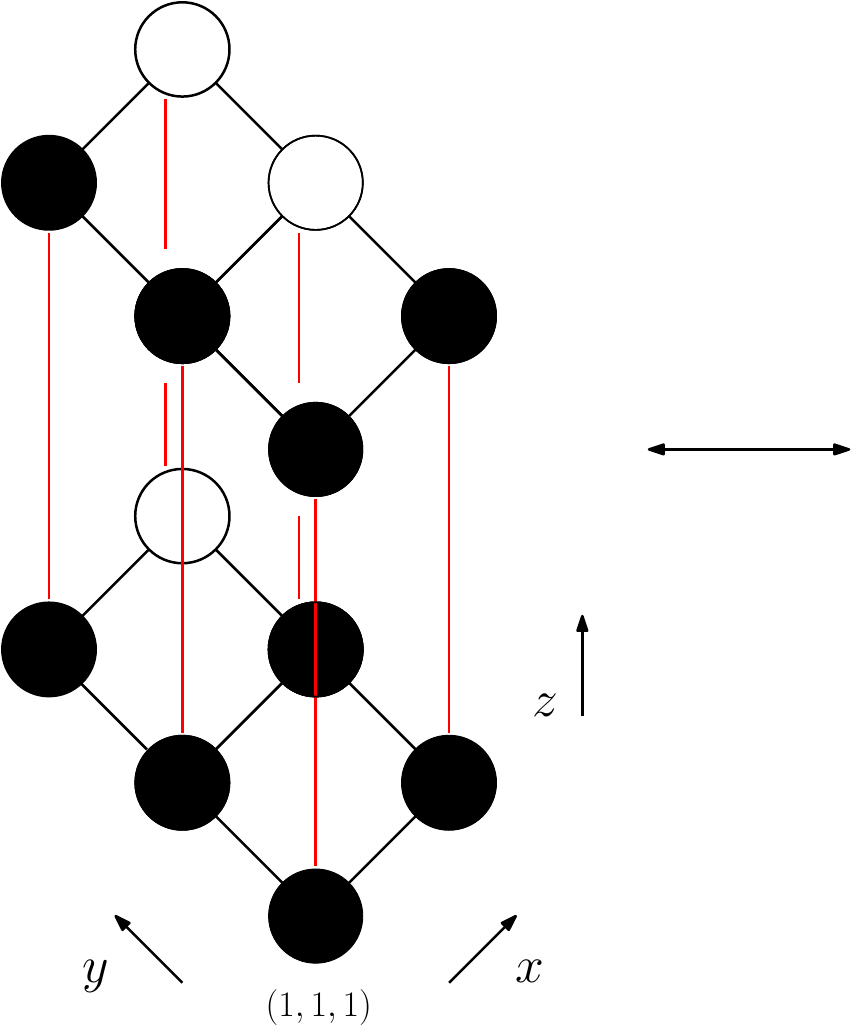}
	\includegraphics[width=.26\linewidth]{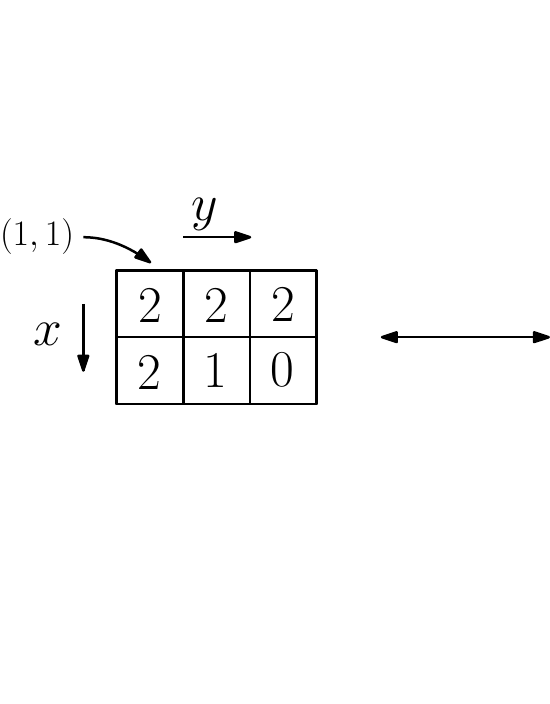}
	\includegraphics[width=.26\linewidth]{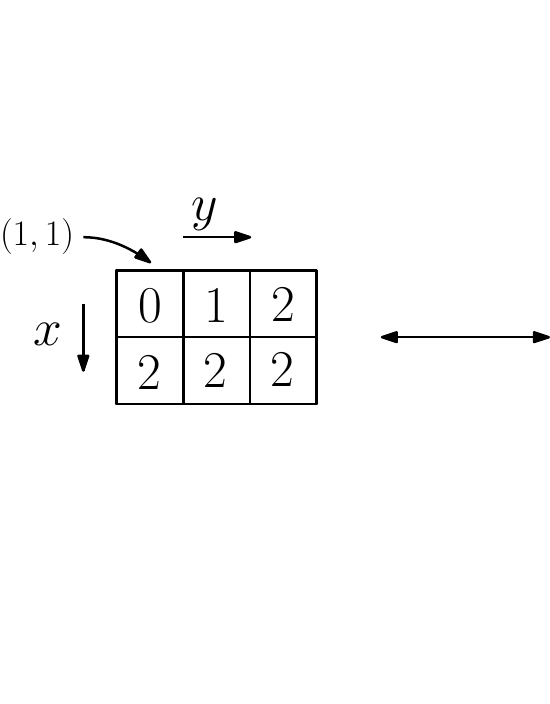}
	\includegraphics[width=.15\linewidth]{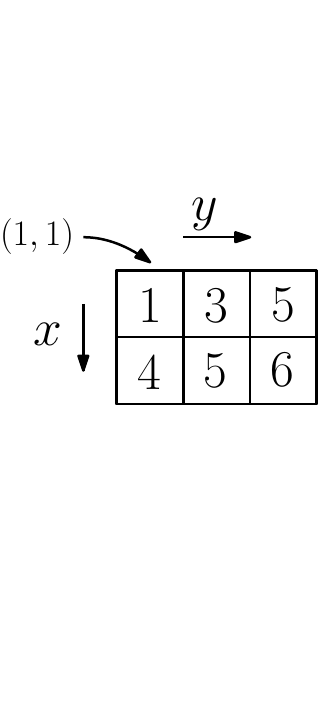}
	\caption{An example of the map $\Psi$ described in Definition \ref{def:prodincbij}.}
	\label{fig:tableaubijectionexample}
\end{figure}

\begin{definition}{\cite[Section 1]{Pechenik2014}}
	\label{def:tabjdt}
	Let $T \in \inc{q}{\lambda}$. Delete all labels 1
	from $T$. Consider the set of boxes that are either empty or contain 2. We simultaneously delete
	each label 2 that is adjacent to an empty box and place a 2 in each empty box that is adjacent to a 2. Now consider the set of boxes that are either empty or contain 3, and repeat the above
	process. Continue until all empty boxes are located at outer corners of $\lambda$.
	Finally, label those boxes $q + 1$ and then subtract 1 from each entry. The
	result is the \emph{$K$-promotion of $T$}, which we denote $\kpro(T)$. Note that $\kpro(T) \in \inc{q}{\lambda}$.
\end{definition}

Along with defining $\Psi$, Dilks, Pechenik, and Striker also showed that $\Psi$ intertwines $\mathrm{Pro}_{(1,1,-1)}$ and $\kpro$.

\begin{theorem}{\cite[Theorem 4.1, Lemma 4.2]{DPS2017}}
	\label{thm:psibij}
	$\Psi$ is an equivariant bijection between $\prodchains$ under $\mathrm{Pro}_{(1,1,-1)}$ and $\inc{a+b+c-1}{a \times b}$ under $\kpro$.
\end{theorem}

Furthermore, we can relate the cardinality of $I$ to the sum of the entries in $\Psi(I)$.

\begin{lemma}
	\label{lemma:boxsum}
	If $I \in \prodchainstwo$, the sum of the boxes in $\Psi(I)$ is equal to $f(I)+b(b+2)$ where $f$ is the cardinality statistic.
\end{lemma}
\begin{proof}
	This follows from the definition of $\Psi$ and the shape of $\Psi(I)$.
\end{proof}

As a result of this lemma, if we can find an appropriate homomesy result on increasing tableaux, we can transfer the result over to $\prodchainstwo$ under Pro$_{(1,1,-1)}$ using $\Psi$, then to $\prodchainstwo$ under Row using Corollary \ref{cor:rowrecomb}. As it turns out, the appropriate homomesy result has already been discovered by J.\ Bloom, Pechenik, and D.\ Saracino.

\begin{theorem}{\cite[Theorem 6.5]{BPS2016}}
	\label{thm:tabhom}
	Consider an increasing tableau of shape $2 \times n$ and let $S$ be a subset of boxes fixed under $180\degree$ rotation. Additionally, let $\sigma_S$ be the statistic of summing the entries in the boxes of $S$.  Then for any $q$, $(\mathrm{Inc}^q(\lambda)$, $K$-$\mathrm{Pro}$, $\sigma_S)$ is homomesic.
\end{theorem}

Note that if $S$ consists of all boxes in a $2 \times n$ increasing tableau, then $S$ is fixed under 180$\degree$ rotation. Moreover, for $I \in \prodchainstwo$, $\Psi(I)$ is an increasing tableau of shape $2 \times b$. With this theorem, we now have sufficient machinery to prove Theorem \ref{thm:mainhom}.

\begin{proof}[Proof of Theorem \ref{thm:mainhom}]
	Each orbit of $J([2] \times [b] \times [c])$ under $\mathrm{Pro}_{(1,1,-1)}$ corresponds to an orbit of $\mathrm{Inc}^{b+c+1}(\lambda)$ under $K$-Pro. Therefore, by using Theorem \ref{thm:tabhom}, Lemma \ref{lemma:boxsum}, Theorem  \ref{thm:psibij} and the map $\Psi$, we may already conclude $\prodchainstwo$ exhibits homomesy under Pro$_{(1,1,-1)}$.  Moreover, Pro$_{(-1,-1,1)}$ reverses the direction that our hyperplanes sweep through our poset, which merely reverses our orbits of order ideals. Thus, $\prodchainstwo$ also exhibits homomesy under Pro$_{(-1,-1,1)}$. To prove Theorem \ref{thm:mainhom} for the remaining $v$, we begin with $v=(1,1,1)$, which is Row.
	
	Let $\orb_1, \orb_2$ be orbits of $\prodchainstwo$ under Row. Additionally, let $R_1$ and $R_2$ be the orbits formed by applying recombination $\recomb{(1,1,1)}{3}$ to each order ideal $I$ in $\orb_1$ and $\orb_2$, respectively. By Corollary \ref{cor:rowrecomb}, we have that $R_1$ and $R_2$ are orbits under Pro$_{(1,1,-1)}$. Therefore, we know that the average of the cardinality over $R_1$ and $R_2$ must be equal. As a result, the average of the cardinality over $\orb_1$ and $\orb_2$ must be equal. Hence, $\prodchainstwo$ is homomesic under Row. Again, because Pro$_{(-1,-1,-1)}$ merely reverses the direction of hyperplane toggles, we conclude that $\prodchainstwo$ is homomesic under Pro$_{(-1,-1,-1)}$.

	We now turn our attention to Pro$_{(-1,1,-1)}$ and Pro$_{(1,-1,-1)}$. Using Corollaries \ref{cor:prorecomb} and \ref{cor:thirdrecomb}, Remark \ref{rmk:recombbij}, and similar arguments as above, we see $\prodchainstwo$ is homomesic under both Pro$_{(-1,1,-1)}$ and Pro$_{(1,-1,-1)}$.  Pro$_{(1,-1,1)}$ and Pro$_{(-1,1,1)}$ reverse the orbits of Pro$_{(-1,1,-1)}$ and Pro$_{(1,-1,-1)}$, respectively, so $\prodchainstwo$ is homomesic under both Pro$_{(1,-1,1)}$ and Pro$_{(-1,1,1)}$ as well.
	
	We have shown the desired triples are homomesic, but we still must show the orbit average is $bc$. Due to rotational symmetry, the order filters of $\prodchainstwo$ are in bijection with the order ideals of $\prodchainstwo$. More specifically, let $I\in \prodchainstwo$. Let $H \in \prodchainstwo$ be the order ideal isomorphic to $P\setminus I$. Therefore, $f(I)+f(H)=2bc$. As a result, we can say the global average of $f$ is $bc$, and hence the triple must be $bc$-mesic.
\end{proof}

We immediately obtain the following corollaries by symmetry.
\begin{corollary}
	\label{cor:maincor1}
	Let $f$ be the cardinality statistic and $v \in \{\pm 1\}^n$.  The triple $(J([a] \times [2] \times [c]), \mathrm{Pro}_v, f)$ is $ac$-mesic and the triple $(J([a] \times [b] \times [2]), \mathrm{Pro}_v, f)$ is $ab$-mesic.
\end{corollary}

\begin{proof}[Proof of Corollary \ref{cor:maincor1}]
	Given an orbit $\orb$ of $J([a] \times [2] \times [c])$ under Pro$_{v}$, we can use a cyclic rotation of coordinates and appropriate choice of $v'$ to obtain an orbit $\orb'$ of $\prodchainstwo$ under Pro$_{v'}$ such that $\orb$ and $\orb'$ are in bijection. A similar argument applies to $J([a] \times [b] \times [2])$.
\end{proof}

We conclude the section by determining that Theorem \ref{thm:mainhom} does not generalize to an arbitrary product of three chains, a product of four chains, or a product of arbitrarily many two-element chains. Homomesy holds on the poset $[3] \times [3] \times [3]$; however, it does not on $[3] \times [3] \times [4]$.

\begin{proposition}
	\label{prop:334}
	Let $f$ be the cardinality statistic and $v \in \{\pm 1\}^n$. The triple $(J([3] \times [3] \times [3]), \mathrm{Pro}_v, f)$ is $27/2$-mesic. However, the triple $(J([3] \times [3] \times [4]), \mathrm{Pro}_v, f)$ is not homomesic.
\end{proposition}
\begin{proof}
	A calculation using SageMath \cite{sage} shows that $J([3] \times [3] \times [3])$ under Row with statistic $f$ has 124 orbits, all with average 27/2. However, $J([3] \times [3] \times [4])$ under Row with statistic $f$ has 456 orbits with average 18, 2 orbits with average $161/9 \approx 17.89$, and 2 orbits with average $163/9 \approx 18.11$. Using recombination, we obtain the same result for any $\mathrm{Pro}_v$.
\end{proof}

We can further inquire about homomesy in higher dimensions. We find homomesy in the poset $[2] \times [2] \times [2] \times [2]$, but a negative result if any of the chains have size three.  If we use only chains of size two, homomesy fails in dimension five.

\begin{proposition}
	\label{prop:22222}
	Let $f$ be the cardinality statistic and $v \in \{\pm 1\}^n$. The triple $J([2] \times [2] \times [2] \times [2]), \mathrm{Pro}_v, f)$ is 8-mesic. However, the triple $(J([2] \times [2] \times [2] \times [3]), \mathrm{Pro}_v, f)$ is not homomesic.  Additionally, the triple $(J([2] \times [2] \times [2] \times [2] \times [2]), \mathrm{Pro}_v, f)$ is not homomesic.
\end{proposition}
\begin{proof}
	A calculation using SageMath \cite{sage} shows that $J([2] \times [2] \times [2] \times [2])$ under Row with statistic $f$ has 36 orbits, all with average 8. However, $J([2] \times [2] \times [2] \times [3])$ has 109 orbits with average 12, 6 orbits with average $82/7 \approx 11.71$, and 6 orbits with average $86/7 \approx 12.29$. Additionally, $J([2] \times [2] \times [2] \times [2] \times [2])$ has 771 orbits with average 16, 60 orbits with average $115/7 \approx 16.43$, 60 orbits with average $109/7 \approx 15.57$, 30 orbits with average $61/4=15.25$, 30 orbits with average $67/4=16.75$, 6 orbits with average 11, and 6 orbits with average 21. Using recombination, we once again obtain the same results for any $\mathrm{Pro}_v$.
\end{proof}

\section{Tableaux and Refined Results}
In this section, we prove several related results and corollaries. Although Proposition \ref{prop:334} shows the cardinality statistic fails to be homomesic for an arbitrary product of three chains, Corollary \ref{cor:gen3chain} gives us a subset within the product of three chains that does exhibit homomesy. Additionally, we use our main homomesy result to obtain a new homomesy result on increasing tableaux in Corollary \ref{cor:tabhom}. In Corollary \ref{cor:rotsym}, we use refined homomesy results on increasing tableaux to state more refined homomesy results on order ideals.

For our main homomesy result, we used the bijection $\Psi^{-1}$ to translate a homomesy result on increasing tableaux to order ideals of a product of chains poset. After rotation on our product of chains to obtain Corollary \ref{cor:maincor1}, we can translate back to increasing tableaux using $\Psi$ to obtain an additional homomesy result on increasing tableaux. This is in the same spirit as the tri-fold symmetry used by Dilks, Pechenik, and Striker \cite[Corollary 4.7]{DPS2017}.

\begin{corollary}
	\label{cor:tabhom}
	Let $\lambda$ be an $a \times b$ rectangle and let $\sigma_\lambda$ be the statistic of summing the entries in the boxes of $\lambda$.  Then $(\mathrm{Inc}^{a+b+1}(\lambda)$, $K$-$\mathrm{Pro}$, $\sigma_\lambda)$ is $ab+\frac{ab(a+b)}{2}$-mesic.
\end{corollary}

\begin{proof}
	Each orbit of $\mathrm{Inc}^{a+b+1}(\lambda)$ under $K$-Pro corresponds to an orbit of $J([a] \times [b] \times [2])$ under $\mathrm{Pro}_{(1,1,-1)}$. For each $I \in J([a] \times [b] \times [2])$, $\sigma_\lambda(\Psi(I))=f(I)+\frac{ab(a+b)}{2}$ where $f$ is the cardinality statistic. Applying Corollary \ref{cor:maincor1}, the result follows.
\end{proof}

Additionally, we have a more refined homomesy result of Theorem \ref{thm:mainhom}. We obtain this using the rotational symmetry condition of Theorem \ref{thm:tabhom}. Define the columns $L^{j,k}_{1,2}=\{(i_1,i_2,i_3)\in [2] \times [b] \times [c] \ | \ i_1=j, i_2=k\}$. This notation is similar to the layer notation of Definition \ref{def:layers} with the exception that we fix two coordinates instead of one. We also define antipodal elements in a poset $[a]\times [b]$ to better describe the rotational symmetry.

\begin{definition}
	\label{def:axbantipodal}
	Let $P = [a] \times [b]$. If $x=(x_1,x_2)$ and $y=(a+1-x_1,b+1-x_2)$, then $x$ and $y$ are \emph{antipodal} in $P$.
\end{definition}

\begin{corollary}
	\label{cor:rotsym}
	Let $L^{j_1,k_1}_{1,2}$ and $L^{j_2,k_2}_{1,2}$ be such that the coordinates $(j_1,k_1)$ and $(j_2,k_2)$ are antipodal in $[2] \times [a]$. If $f_L(I)$ denotes the cardinality of $I$ on $L^{j_1,k_1}_{1,2}$ and $L^{j_2,k_2}_{1,2}$, then for $v \in \{\pm 1\}^n$, $(\prodchainstwo, \mathrm{Pro}_v, f_L)$ is $c$-mesic.
\end{corollary}
\begin{proof}
	The antipodal coordinates $(j_1,k_1)$ and $(j_2,k_2)$ are chosen so that the columns $L^{j_1,k_1}_{1,2}$ and $L^{j_2,k_2}_{1,2}$ correspond to a set of boxes in an increasing tableau fixed under $180^{\circ}$ rotation. In other words, we can use the refined homomesy result on increasing tableaux from Theorem \ref{thm:tabhom} and translate to $\prodchainstwo$ using the bijection $\Psi^{-1}$. As a result, we know $(\prodchainstwo, \mathrm{Pro}_v, f_L)$ exhibits homomesy. What remains to be shown is that the triple is $c$-mesic. Due to rotational symmetry, the order filters of $[2] \times [b] \times [c]$ are in bijection with the order ideals of $[2] \times [b] \times [c]$. More specifically, let $I\in \prodchainstwo$. Let $H \in \prodchainstwo$ be the order ideal isomorphic under rotation to the order filter $P\setminus I$. Therefore, $f_L(I)+f_L(H)=2c$. As a result, we can say the global average of $f_L$ is $c$. This gives us that $(\prodchainstwo, \mathrm{Pro}_{(1,1,-1)}, f_L)$ is $c$-mesic; using recombination we obtain that $(\prodchainstwo, \mathrm{Pro}_v, f_L)$ is $c$-mesic.
\end{proof}

\begin{example}
	We demonstrate Corollary \ref{cor:rotsym} using the poset $[2] \times [2] \times [2]$. Figure \ref{fig:rotsymcolumns} highlights in red two columns in our poset, $L^{1,2}_{1,2}$ and $L^{2,1}_{1,2}$. Because $(1,2)$ and $(2,1)$ are antipodal elements in the poset $[2] \times [2]$, we can apply Corollary \ref{cor:rotsym} to these two columns. Note that the corollary is valid for any $\mathrm{Pro}_v$. We show an example orbit of $\row$ in Figure \ref{fig:rotsymcolumns_orbit}. The orbit we chose has size five. Additionally, if we sum the cardinality of the order ideals in the columns $L^{1,2}_{1,2}$ and $L^{2,1}_{1,2}$ over the entire orbits, we obtain ten. Therefore, the average over the entire orbit is $10/5=2$, which is the value $c$ when expressing the poset in the form $[2] \times [b] \times [c]$. Additionally, if we select any other orbit, the corollary tells us we will obtain an average of 2.
\end{example}

\begin{figure}[htbp]
	\includegraphics[width=.25\linewidth]{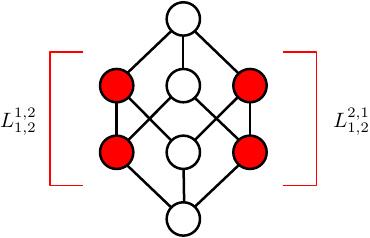}
	\caption{Because $(1,2)$ and $(2,1)$ are antipodal in the poset $[2] \times [2]$, we can apply Corollary \ref{cor:rotsym} to the columns $L^{1,2}_{1,2}$ and $L^{2,1}_{1,2}$, shown here in red.}
	\label{fig:rotsymcolumns}
\end{figure}

\begin{figure}[htbp]
	\includegraphics[width=.25\linewidth]{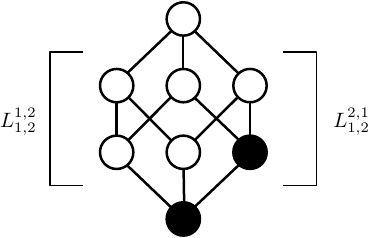} \quad
	\includegraphics[width=.25\linewidth]{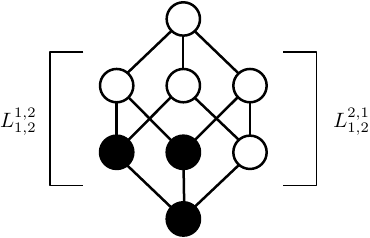} \quad
	\includegraphics[width=.25\linewidth]{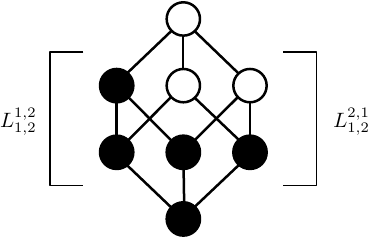}
	\includegraphics[width=.25\linewidth]{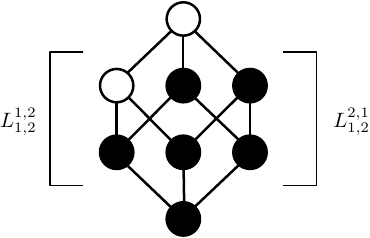} \quad
	\includegraphics[width=.25\linewidth]{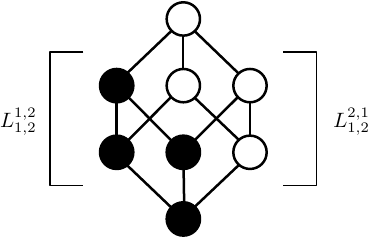}
	\caption{These five order ideals form an orbit under rowmotion. By summing the cardinality of the order ideal in just columns $L^{1,2}_{1,2}$ and $L^{2,1}_{1,2}$ and dividing by the size of the orbit, we obtain $10/5=2$, which corresponds to $c$ in $[2] \times [b] \times [c]$.}
	\label{fig:rotsymcolumns_orbit}
\end{figure}

Pechenik further generalized the results of \cite{BPS2016} and the result stated in Theorem \ref{thm:tabhom}. From this, we obtain a more general analogue of Corollary \ref{cor:rotsym}. We summarize the relevant definition and theorem below.

\begin{definition}{\cite[Section 1]{Pechenik2017}}
	The \textit{frame} of a partition $\lambda$ is the set \textbf{Frame}($\lambda$) of all boxes in the first or last row, or in the first or last column.
\end{definition}

\begin{figure}[htbp]
	\includegraphics[width=.25\linewidth]{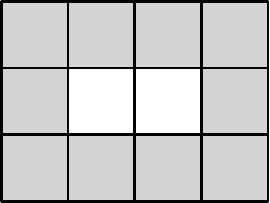} \quad
	\caption{This is the partition of shape $3 \times 4$. The frame of the partition is the set of boxes highlighted in gray.}
	\label{fig:frame}
\end{figure}

\begin{theorem}{\cite[Theorem 1.6]{Pechenik2017}}
	\label{thm:frames}
	Let $S$ be a subset of \textbf{Frame}($m \times n$) that is fixed under $180\degree$ rotation. Then $(\mathrm{Inc}^q(m \times n)$, $K$-$\mathrm{Pro}$, $\sigma_S)$ is $\frac{(q+1)|S|}{2}$-mesic.
\end{theorem}

The following is a new corollary of Theorem \ref{thm:frames}. It uses the bijection $\Psi^{-1}$ and techniques similar to those of Corollary \ref{cor:rotsym} to prove a more general analogue of Corollary \ref{cor:rotsym} in the product of three chains.

\begin{corollary}
	\label{cor:gen3chain}
	Let $P=[a] \times [b] \times [c]$. Additionally, let $L^{j_1,k_1}_{1,2}$ and $L^{j_2,k_2}_{1,2}$ be such that the coordinates $(j_1,k_1)$ and $(j_2,k_2)$ are antipodal in $[a] \times [b]$, each $j_i$ is $1$ or $a$, and each $k_i$ is $1$ or $b$. If $f_L(I)$ denotes the cardinality of $I$ on $L^{j_1,k_1}_{1,2}$ and $L^{j_2,k_2}_{1,2}$, then for $v \in \{\pm 1\}^n$, $(J([a]\times[b]\times[c]),\mathrm{Pro}_{v}, f_L)$ is $c$-mesic.
\end{corollary}
\begin{proof}
	Similarly to the proof of Corollary \ref{cor:rotsym}, the antipodal coordinates $(j_1,k_1)$ and $(j_2,k_2)$ are chosen so that the columns $L^{j_1,k_1}_{1,2}$ and $L^{j_2,k_2}_{1,2}$ correspond to a set of boxes in an increasing tableau fixed under $180^{\circ}$ rotation. Additionally, the columns correspond to boxes in the frame of the tableau. As a result, we know $(J([a]\times[b]\times[c]),\mathrm{Pro}_v, f_L)$ exhibits homomesy by translating the refined homomesy result on increasing tableaux from Theorem \ref{thm:frames} to $J([a]\times[b]\times[c])$ using the bijection $\Psi^{-1}$. We must now show that the triple is $c$-mesic. Due to rotational symmetry, the order filters of $P$ are in bijection with the order ideals of $P$. Let $I \in J(P)$ and let $H \in J(P)$ be the order ideal isomorphic under rotation to the order filter $P \setminus I$. Because the two columns $L^{j_1,k_1}_{1,2}$ and $L^{j_2,k_2}_{1,2}$ each contain $c$ elements, $f_L(I)+f_L(H)=2 c$. Therefore, the global average of $f_L$ is $c$ and as a result, the triple is $c$-mesic. This gives the result for $v=(1,1,-1)$; using recombination we obtain the result for all $v$.
\end{proof}

Most of our results have required a chain of size two. However, Corollary \ref{cor:gen3chain} applies to an arbitrary product of three chains, but to antipodal columns on the ``outside" of the poset.

\begin{example}
	Consider the product of chains $[3] \times [4] \times [2]$ in Figure \ref{fig:frameexample}. Note that $(2,1)$ and $(2,4)$ are antipodal in the poset $[3] \times [4]$. Also note that the red columns $L^{2,1}_{1,2}$ and $L^{2,4}_{1,2}$ correspond to boxes in the frame of the partition of shape $3 \times 4$. As a result, Corollary \ref{cor:gen3chain} can be applied, which says if we take any orbit of $J([3]\times [4] \times [2])$ under $\mathrm{Pro}_{v}$, the average of the cardinality of columns $L^{2,1}_{1,2}$ and $L^{2,4}_{1,2}$ over the orbit will be 2.
\end{example}

\begin{figure}[htbp]
	\begin{center}
		\includegraphics[width=.3\linewidth]{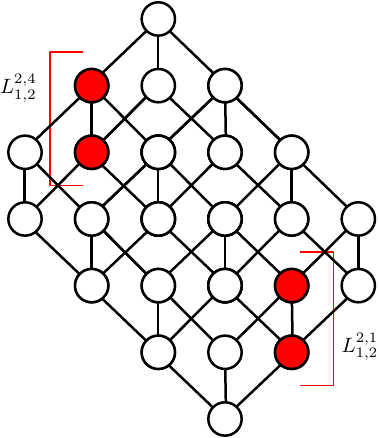} \hspace{.75in}
		\includegraphics[width=.5\linewidth]{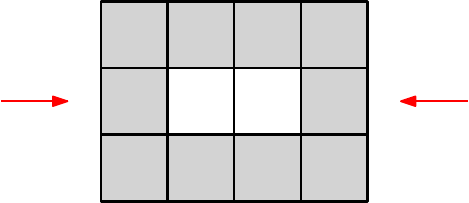}
		\caption{The poset elements $(2,1)$ and $(2,4)$ are antipodal in $[3] \times [4]$. Additionally, the columns $L^{2,1}_{1,2}$ and $L^{2,4}_{1,2}$ correspond to boxes in the frame of the partition of shape $3 \times 4$, as shown by the figure on the right.}
		\label{fig:frameexample}
	\end{center}
\end{figure}

\section{Beyond the product of chains}
We opted to state our recombination results in Section 4 for the product of chains rather than in full generality in order to emphasize the important aspects of the proofs without further complicating the notation. We now generalize the recombination technique from a product of chains to any ranked poset. We begin by presenting several previous definitions in greater generality.

\begin{definition}{\cite[Definition 3.13]{DPS2017}}
	We say that an \emph{$n$-dimensional lattice projection} of a ranked poset $P$ is an order and rank preserving map $\pi : P \rightarrow \mathbb{Z}^n$, where the rank function on $\mathbb{Z}^n$ is the sum of the coordinates and $x \le y$ in $\mathbb{Z}^n$ if and only if the componentwise difference $y-x$ is in $(\mathbb{Z}_{\ge 0})^n$.
\end{definition}
\begin{definition}{\cite[Definition 3.14]{DPS2017}}
	Let $P$ be a poset with an $n$-dimensional lattice projection $\pi$ and let $v  \in \{\pm 1\}^n$.  Let $T_{\pi, v}^i$ be the product of toggles $t_x$ for all elements $x$ of $P$ that lie on the affine hyperplane $\langle \pi(x),v \rangle=i$.  If there is no such $x$, then this is the empty product, considered to be the identity.  Define \textit{promotion with respect to $\pi$ and $v$} as the (finite) toggle product Pro$_{\pi,v}=\dots T_{\pi,v}^{-2} T_{\pi,v}^{-1} T_{\pi,v}^{0} T_{\pi,v}^{1} T_{\pi,v}^{2}\dots$
\end{definition}

For $P$ with $n$-dimensional lattice projection, we generalize the definition of a layer from Definition \ref{def:layers} using the lattice projection $\pi(P)$. More specifically, because $\pi(P) \in \mathbb{Z}^n$, we  use our notion of layers on a product of chains and the preimage of $\pi$ to define layers on $P$.

\begin{definition}
	\label{def:genlayers}
	Let $P$ be a poset with $n$-dimensional lattice projection $\pi$. Define the $j$th $\gamma$-layer of $P$ as \[L_{\gamma}^j=\{\pi^{-1}(i_1,i_2,\ldots, i_n) \ | \ i_{\gamma}=j \text{ and } (i_1,i_2,\ldots, i_n)\in \mathbb{Z}^n\}\] and
	the $j$th $\gamma$-layer of $I\in J(P)$ as \[L_{\gamma}^j(I)=L_{\gamma}^j \cap I.\]
	Additionally, given $L_{\gamma}^j$ and $L_{\gamma}^j(I)$, we abuse notation to define \[(L_{\gamma}^{j})^{\widehat{\gamma}}=\{\pi^{-1}((i_1,i_2,\ldots, i_n)^{\widehat{\gamma}}) \ | \ i_{\gamma}=j \text{ and } (i_1,i_2,\ldots, i_n)\in \mathbb{Z}^n\},\]
	\[L_{\gamma}^j(I)^{\widehat{\gamma}}=(L_{\gamma}^{j})^{\widehat{\gamma}} \cap I,\]
	where $\pi^{-1}((i_1,i_2,\ldots, i_n)^{\widehat{\gamma}})$ denotes forming the poset given by the preimage of the $(n-1)$-dimensional poset obtained from deleting the coordinate $\gamma$ and $(L_{\gamma}^{j})^{\widehat{\gamma}} \cap I$ denotes using elements in the order ideal $I$ to form an order ideal with the corresponding elements in $(L_{\gamma}^{j})^{\widehat{\gamma}}$.
\end{definition}

In order to prove results regarding recombination in Section 4, we relied heavily on the ability to commute the toggles of promotion. More specifically, we showed that any promotion could be thought of as sequence of $(n-1)$-dimensional promotions on the layers of our product of chains. We introduce the notation for an analogous result.

\begin{definition}
	Let $P$ be a poset with $n$-dimensional lattice projection $\pi$, $v \in \{\pm 1\}^n$, and $\gamma \in [n]$. We define $T^{j}_{\mathrm{Pro}_{\pi,v^{\widehat{\gamma}}}}$ as the toggle product of Pro$_{\pi,v^{\widehat{\gamma}}}$ on $(L_{\gamma}^j)^{\widehat{\gamma}}$.
\end{definition}

This definition allows us to perform an $(n-1)$-dimensional promotion on a single layer of $P$. Before we give a general definition of recombination, we present a higher dimensional analogue of a result of Striker and Williams. In \cite[Theorem 5.4]{SW2012}, they found a conjugating toggle element; in other words, the toggles necessary to state a explicit bijection from $J(P)$ under $\row^{-1}$ to $J(P)$ under Pro using conjugation. We determine conditions on $v$ and $w$ such that we can find a conjugating toggle element to conjugate from $J(P)$ under $\mathrm{Pro}_{\pi,v}$ to $J(P)$ under $\mathrm{Pro}_{\pi,w}$.

\begin{theorem}
	\label{thm:conjtoggle}
	Let $P$ be a poset with $n$-dimensional lattice projection $\pi$ with $v, w \in \{\pm 1\}^n$ such that $v_\gamma = 1,w_\gamma = -1$, and $v^{\widehat{\gamma}}=w^{\widehat{\gamma}}$. There exists an equivariant bijection between $J(P)$ under $\mathrm{Pro}_{\pi,v}$ and $\mathrm{Pro}_{\pi,w}$ given by acting on an order ideal by $D_\gamma=\prod_{i=1}^{a_\gamma-1} \prod_{j=1}^{i} (T^{i+1-j}_{\mathrm{Pro}_{\pi,v^{\widehat{\gamma}}}})^{-1}$ where $L_{\gamma}^{a_\gamma}$ is the maximum non-empty layer in $P$.
\end{theorem}
\begin{proof}
	Without loss of generality, $v_\gamma = 1$ and $w_\gamma = -1$. As a result, $\mathrm{Pro}_{\pi,w} = \prod_{i=1}^{a_\gamma} T_{\mathrm{Pro}_{\pi,w^{\widehat{\gamma}}}}^{a_\gamma+1-i}$ and $\mathrm{Pro}_{\pi,v} = \prod_{i=1}^{a_\gamma} T_{\mathrm{Pro}_{\pi,v^{\widehat{\gamma}}}}^i$. Note that $w^{\widehat{\gamma}}=v^{\widehat{\gamma}}$. We will commute toggles to show $\mathrm{Pro}_{\pi,w} D_\gamma = D_\gamma \mathrm{Pro}_{\pi,v}$. When we expand, we obtain 
	\begin{align*}
	\mathrm{Pro}_{\pi,w} D_\gamma=& T_{\mathrm{Pro}_{\pi,w^{\widehat{\gamma}}}}^{a_\gamma}T_{\mathrm{Pro}_{\pi,w^{\widehat{\gamma}}}}^{a_\gamma-1}\dots T_{\mathrm{Pro}_{\pi,w^{\widehat{\gamma}}}}^{1}(T_{\mathrm{Pro}_{\pi,w^{\widehat{\gamma}}}}^{1})^{-1}(T_{\mathrm{Pro}_{\pi,w^{\widehat{\gamma}}}}^{2})^{-1}(T_{\mathrm{Pro}_{\pi,w^{\widehat{\gamma}}}}^{1})^{-1}\dots (T_{\mathrm{Pro}_{\pi,w^{\widehat{\gamma}}}}^{a_\gamma-1})^{-1} \\&(T_{\mathrm{Pro}_{\pi,w^{\widehat{\gamma}}}}^{a_\gamma-2})^{-1}\dots (T_{\mathrm{Pro}_{\pi,w^{\widehat{\gamma}}}}^{1})^{-1}
	\end{align*}
	and
	\begin{align*}
	D_\gamma \mathrm{Pro}_{\pi,v}=&(T_{\mathrm{Pro}_{\pi,w^{\widehat{\gamma}}}}^{1})^{-1}(T_{\mathrm{Pro}_{\pi,w^{\widehat{\gamma}}}}^{2})^{-1}(T_{\mathrm{Pro}_{\pi,w^{\widehat{\gamma}}}}^{1})^{-1}\dots (T_{\mathrm{Pro}_{\pi,w^{\widehat{\gamma}}}}^{a_\gamma-1})^{-1}(T_{\mathrm{Pro}_{\pi,w^{\widehat{\gamma}}}}^{a_\gamma-2})^{-1}\dots (T_{\mathrm{Pro}_{\pi,w^{\widehat{\gamma}}}}^{1})^{-1}\\ &T_{\mathrm{Pro}_{\pi,w^{\widehat{\gamma}}}}^{1}T_{\mathrm{Pro}_{\pi,w^{\widehat{\gamma}}}}^{2}\dots T_{\mathrm{Pro}_{\pi,w^{\widehat{\gamma}}}}^{a_\gamma}\\
	=&(T_{\mathrm{Pro}_{\pi,w^{\widehat{\gamma}}}}^{1})^{-1}(T_{\mathrm{Pro}_{\pi,w^{\widehat{\gamma}}}}^{2})^{-1}(T_{\mathrm{Pro}_{\pi,w^{\widehat{\gamma}}}}^{1})^{-1}\dots (T_{\mathrm{Pro}_{\pi,w^{\widehat{\gamma}}}}^{1})^{-1} T_{\mathrm{Pro}_{\pi,w^{\widehat{\gamma}}}}^{a_\alpha}.
	\end{align*}
	
	However, we can commute $T_{\mathrm{Pro}_{\pi,w^{\widehat{\gamma}}}}^{k}$ and $T_{\mathrm{Pro}_{\pi,w^{\widehat{\gamma}}}}^{j}$ or $(T_{\mathrm{Pro}_{\pi,w^{\widehat{\gamma}}}}^{j})^{-1}$ if $|j-k|>1$ because the elements in these toggles could not share a covering relation. Therefore, we can commute toggles of $\mathrm{Pro}_{\pi,w} D_\gamma$ to obtain
	\begin{align*}
	\mathrm{Pro}_{\pi,w} D_\gamma=&(T_{\mathrm{Pro}_{\pi,w^{\widehat{\gamma}}}}^{1})^{-1}(T_{\mathrm{Pro}_{\pi,w^{\widehat{\gamma}}}}^{2})^{-1}(T_{\mathrm{Pro}_{\pi,w^{\widehat{\gamma}}}}^{1})^{-1}\dots (T_{\mathrm{Pro}_{\pi,w^{\widehat{\gamma}}}}^{1})^{-1} T_{\mathrm{Pro}_{\pi,w^{\widehat{\gamma}}}}^{a_\alpha}.
	\end{align*}
	Therefore, $\mathrm{Pro}_{\pi,w} D_\gamma = D_\gamma \mathrm{Pro}_{\pi,v}$ and so $\mathrm{Pro}_{\pi,v} = (D_\gamma)^{-1} \mathrm{Pro}_{\pi,w} D_\gamma$.
\end{proof}

We now present our generalized definition of recombination with respect to an $n$-dimensional lattice projection.

\begin{definition}
	\label{def:latprojrecomb}
	Let $P$ be a poset with $n$-dimensional lattice projection $\pi$, $v \in \{\pm 1\}^n$, and $I \in J(P)$. Define $\genrecomb{\pi,v}{\gamma}=\bigcupdot_{j}L_{\gamma}^j(\mathrm{Pro}_{\pi,v}^{j-1}(I))$ where $\gamma \in [n]$. We will call $\genrecomb{\pi,v}{\gamma}$ the \emph{$(\pi,v,\gamma)-$recombination} of $I$. When context is clear, we will suppress the $(\pi,v,\gamma)$.
\end{definition}

The idea is the same as before; we take certain layers from an orbit of promotion to create a new order ideal. We can now state the analogue of Theorem \ref{thm:ndcommute}, our result regarding toggling commutation, whose proof is similar to the proof of Theorem \ref{thm:ndcommute}.

\begin{theorem}
	\label{thm:genndcommute}
	Let P be a poset with lattice projection $\pi$, $v \in \{\pm 1\}^n$, and $\gamma \in [n]$. Then $\mathrm{Pro}_{\pi,v}=\prod_{j=1}^{a_\gamma} T^{\alpha}_{\mathrm{Pro}_{\pi,v^{\widehat{\gamma}}}}$ where \[
	\alpha =
	\begin{cases}
	j   \hfill & \text{if } v_\gamma=1 \\
	a_\gamma+1-j \hfill & \text{if } v_\gamma=-1. \\
	\end{cases}
	\]
\end{theorem}

As in the product of chains setting, we have conditions to determine when generalized recombination gives us an order ideal. The proof is similar to the proof of Lemma \ref{lemma:orderideal} with the inclusion of the lattice projection $\pi$.

\begin{lemma}
	\label{lemma:genorderideal}
	Let $I\in J(P)$. Suppose we have $v \in \{\pm 1\}^n$ and $\gamma$ such that $v_\gamma=1$. Then $\genrecomb{\pi,v}{\gamma}$ is an order ideal of $P$.
\end{lemma}

We can now state our general recombination result, which shows when recombination gives us an equivariant bijection from $J(P)$ under $\mathrm{Pro}_{u}$ to $J(P)$ under $\mathrm{Pro}_{v}$ for any poset $P$ with $n$-dimensional lattice projection. Again, we omit the proof as it is similar to the proof of Theorem \ref{thm:genrecomb} with the inclusion of the lattice projection $\pi$.
\begin{theorem}
	\label{thm:moregenrecomb}
	Let $I\in J(P)$. Suppose we have $u, v \in \{\pm 1\}^n$ and $\gamma$ such that $v_\gamma=1$, $u_\gamma =-1$, and ${v^{\widehat{\gamma}}} = {u^{\widehat{\gamma}}}$. Then $\mathrm{Pro}_{\pi,u}(\genrecomb{\pi,v}{\gamma})=\Delta_{\pi,v}^{\gamma}(\mathrm{Pro}_{\pi,v}(I))$.
\end{theorem}

In \cite{RW2015}, Rush and Wang showed that order ideals of minuscule posets under rowmotion exhibit homomesy. Using this generalized recombination result and our homomesy result on $\prodchainstwo$, we can obtain an additional homomesy result on order ideals of the type B minuscule poset cross a chain of size two. Let $P_n=([n]\times[n])/S_2$ denote a type B minuscule poset; this can be viewed as the left half of $[n]\times[n]$. Additionally, $P_n$ is isomorphic to $J([2]\times[n-1])$. See Figure \ref{fig:B3minuscule} for an example.

\begin{figure}[htbp]
	\centering
	\includegraphics[width=.15\linewidth]{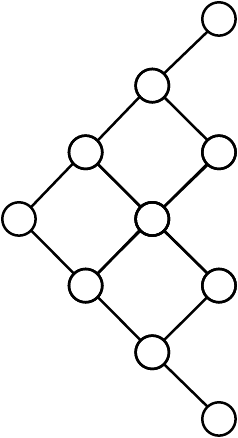}
	\caption{$P_4$, the type B minuscule poset $([4]\times[4])/S_2$}
	\label{fig:B3minuscule}
\end{figure}

\begin{corollary}
	\label{cor:typebmin}
	Let $f$ be the cardinality statistic, $\pi$ be the natural embedding of $P_n \times [2]$ into $\mathbb{Z}^3$, and $v \in \{\pm 1\}^n$. The triple $(J(P_n \times [2]),\mathrm{Pro}_{\pi,v},f)$ is $\frac{n^2+n}{2}$-mesic.
\end{corollary}

\begin{proof}
	Orbits of $J(P_n \times [2])$ under Row are in bijection with orbits of $J([n] \times [n] \times [2])$ under Row where the order ideals are symmetric about the plane $x-y=0$. Let $\orb$ be an orbit of $J(P_n \times [2])$ under Row and $\orb'$ be the orbit of $J([n] \times [n] \times [2])$ in bijection with $\orb$.  We note $\# \orb=\# \orb'$. Let $f(\orb)$ denote the sum of the cardinality of order ideals in $\orb$. By Corollary \ref{cor:maincor1}, $f(\orb)=(\# \orb')n^2$. Alternatively, we can enumerate this sum in $\orb'$ by doubling the cardinality in $\orb$ and removing what is double counted, namely, elements that appear on the plane $x-y=0$. The cardinality of these elements is $(\# \orb')n$ by Corollary \ref{cor:rotsym}. As a result, we have the following equality: $(\# \orb)n^2 = 2f(\orb)-(\# \orb)n$. Rearranging, we get $\frac{f(\orb)}{\# \orb} = \frac{n^2+n}{2}$. Therefore, $(J(P_n \times [2]),\mathrm{Row},f)$ is $\frac{n^2+n}{2}$-mesic. Using the generalized recombination result of Theorem \ref{thm:moregenrecomb}, $(J(P_n \times [2]),\mathrm{Pro}_v,f)$ must be $\frac{n^2+n}{2}$-mesic.
\end{proof}

\begin{example}
	\label{ex:minusculecorollary}
	We demonstrate the proof of Corollary \ref{cor:typebmin} with an example, referring to Figure \ref{fig:minusculecorollary}. The top left order ideal is symmetric about the plane $x-y=0$. When we apply rowmotion, we obtain the top right order ideal, which is still symmetric about the plane $x-y=0$. Because both order ideals are symmetric about $x-y=0$, they are in bijection with the bottom order ideals in Figure \ref{fig:minusculecorollary}. Extending this further, every orbit of $J(P_3 \times [2])$ under rowmotion is in bijection with an orbit of symmetric order ideals of $[3] \times [3] \times [2]$. As a result, we can translate our homomesy result on $J([3] \times [3] \times [2])$ under rowmotion to $J(P_3 \times [2])$ under rowmotion. Recombination gives the homomesy result for all $\mathrm{Pro}_{\pi,v}$.
\end{example}

\begin{figure}[htbp]
	\centering
	\includegraphics[width=.5\linewidth]{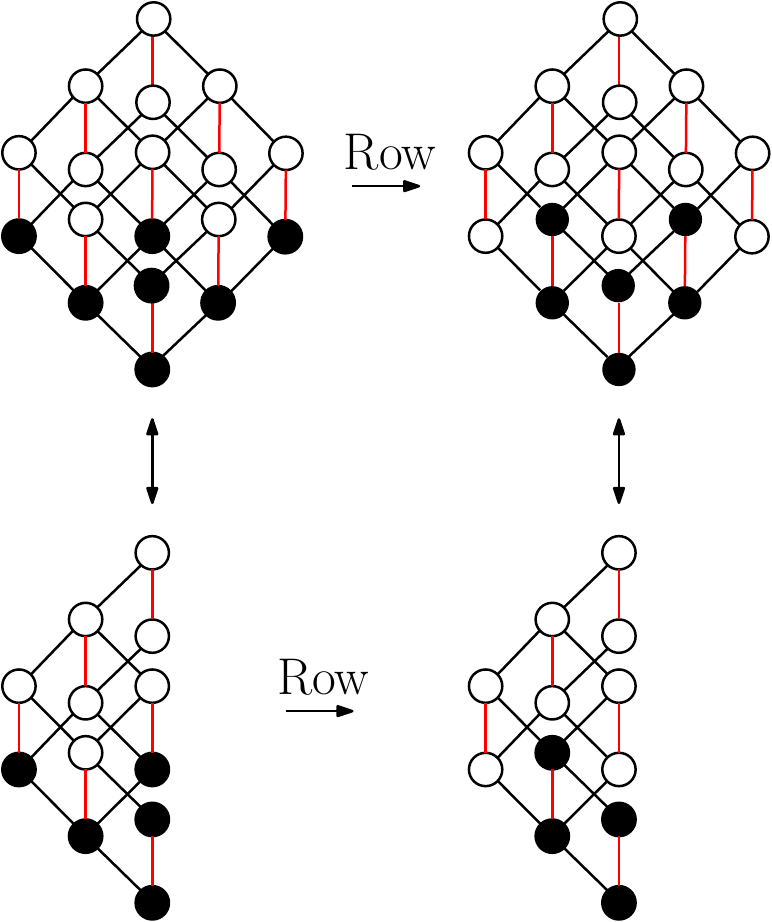}
	\caption{Applying rowmotion to the symmetric order ideal in the top left gives the symmetric order ideal in the top right. These order ideals are in bijection with the bottom order ideals, which are in $P_3 \times [2]$. See Example \ref{ex:minusculecorollary}.}
	\label{fig:minusculecorollary}
\end{figure}

\begin{example}
	\label{ex:latticeproj}
	We now give an example of generalized recombination where we cannot use a simple embedding as our three-dimensional lattice projection. Let our poset be the tetrahedral poset on the left in Figure \ref{fig:asmlatticeproj}; for more on tetrahedral posets, see \cite{Striker2011}. By Proposition 8.5 of \cite{SW2012}, we see the significance of this poset is that its order ideals are in bijection with alternating sign matrices of size $4 \times 4$. We note that this poset cannot be embedded in $\mathbb{Z}^3$ since the element $b$ is covered by four elements. We instead use the lattice projection $\pi$ in Figure \ref{fig:asmlatticeproj}, projecting into $\mathbb{Z}^2$. We note that this lattice projection is not new, as it is used in Figure 18 in \cite{SW2012}. Figure \ref{fig:asmlatticeprojlayers} shows how we will orient this in $\mathbb{Z}^2$.
	
\end{example}
\begin{figure}[htbp]
	\centering
	\includegraphics[width=.7\linewidth]{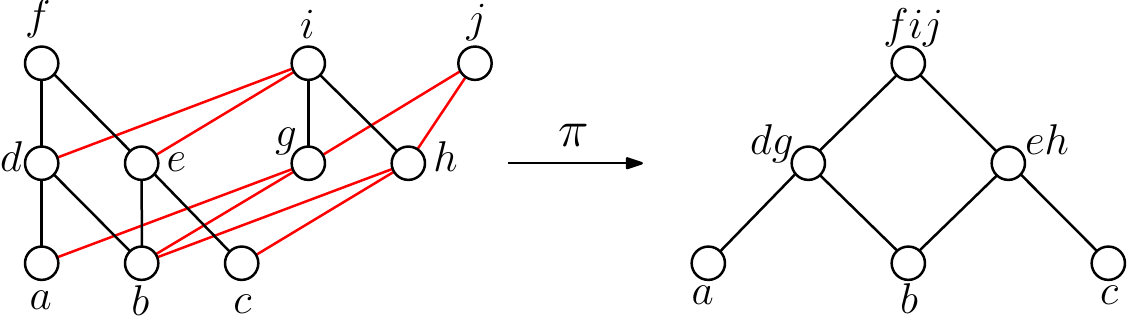}
	\caption{The poset on the left is a tetrahedral poset. For Example \ref{ex:latticeproj}, we will use the lattice projection $\pi$ to the poset on the right.}
	\label{fig:asmlatticeproj}
\end{figure}

Figure \ref{fig:asmroworb} shows a partial orbit under rowmotion. We see from Figure \ref{fig:asmlatticeprojlayers} what our layers are: the first layer consists of $a$, the second layer consists of $b,d,g$, and the third layer consists of $c,e,f,h,i,j$.

\begin{figure}[htbp]
	\centering
	\includegraphics[width=.25\linewidth]{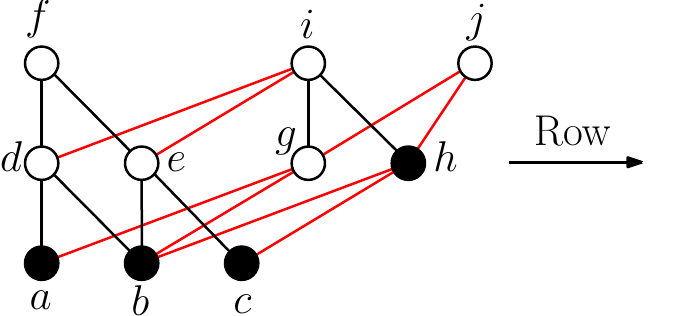}
	\includegraphics[width=.25\linewidth]{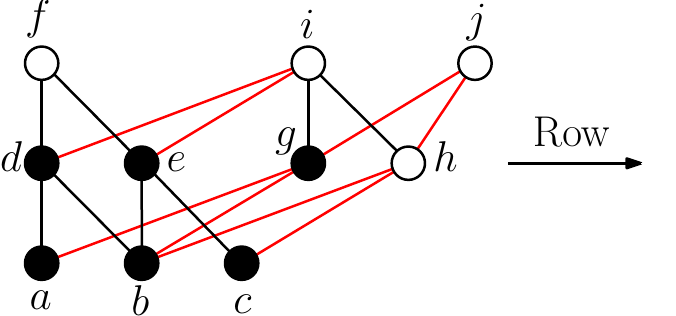}
	\includegraphics[width=.25\linewidth]{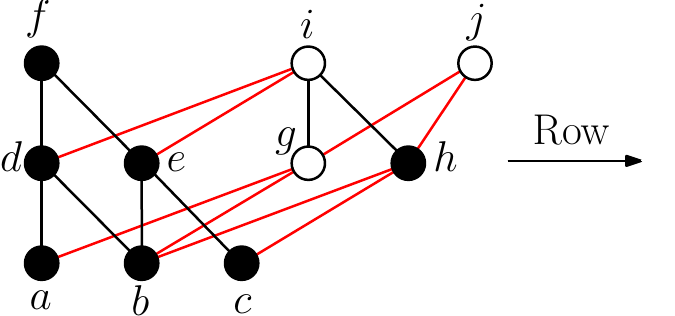}
	\includegraphics[width=.18\linewidth]{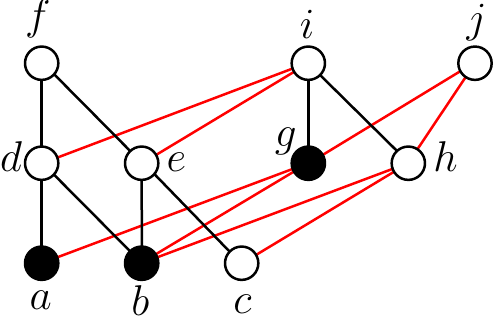}
	\caption{A partial orbit of order ideals under rowmotion. We use this example to demonstrate generalized recombination.}
	\label{fig:asmroworb}
\end{figure}

\begin{figure}[htbp]
	\centering
	\includegraphics[width=.4\linewidth]{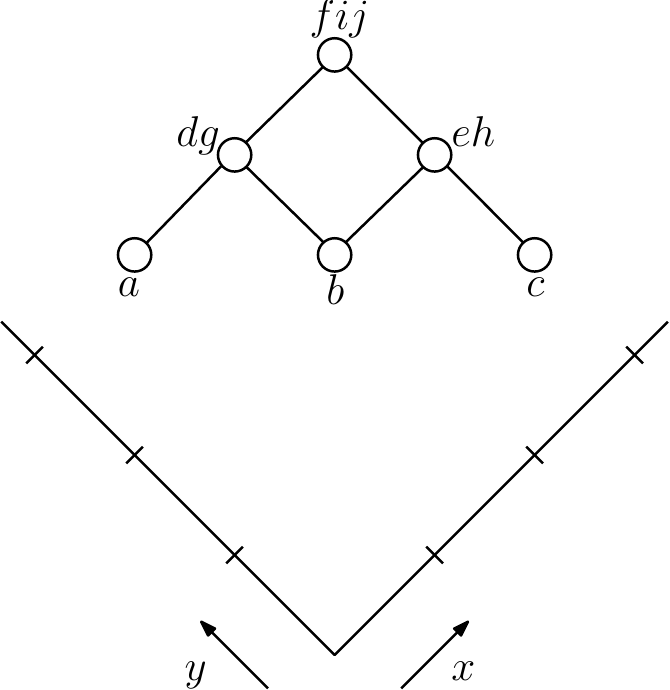}
	\caption{We orient this poset in $\mathbb{Z}^2$ in the following way. Our three layers are the diagonals from $x=1,2,$ and $3$.}
	\label{fig:asmlatticeprojlayers}
\end{figure}

From the partial orbit, we take the first layer from the first order ideal, the second layer from the second order ideal, and the third layer from the third order ideal to form a new order ideal. These are indicated with red in Figures \ref{fig:asmrowrecomborb} and \ref{fig:asmproorb}. We also take the first layer in the second order ideal, the second layer in the third order ideal, and the third layer from the fourth order ideal to form another new order ideal. These are indicated with blue in Figures \ref{fig:asmrowrecomborb} and \ref{fig:asmproorb}. Generalized recombination tells us if we apply promotion to the red order ideal, we should obtain the blue order ideal, which we can see is the case.

\begin{figure}[htbp]
	\centering
	\includegraphics[width=.25\linewidth]{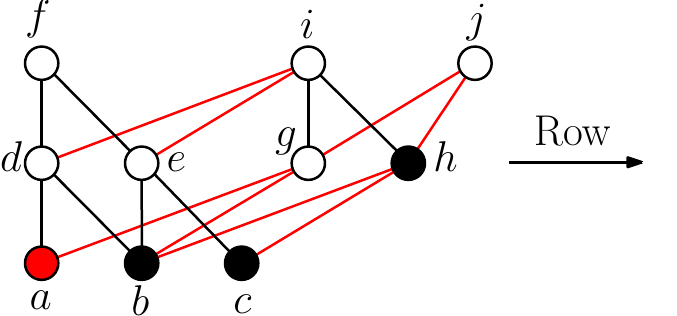}
	\includegraphics[width=.25\linewidth]{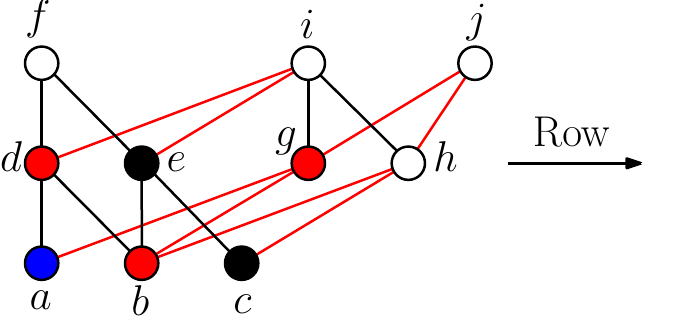}
	\includegraphics[width=.25\linewidth]{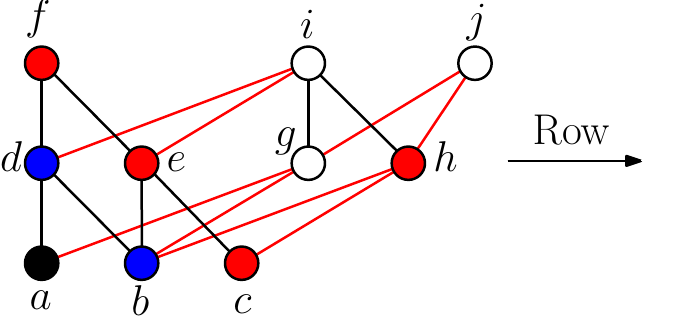}
	\includegraphics[width=.18\linewidth]{asmroworb4.pdf}
	\caption{We use the red layers and blue layers from the partial orbit to form two new order ideals.}
	\label{fig:asmrowrecomborb}
\end{figure}

\begin{figure}[htbp]
	\centering
	\includegraphics[width=.25\linewidth]{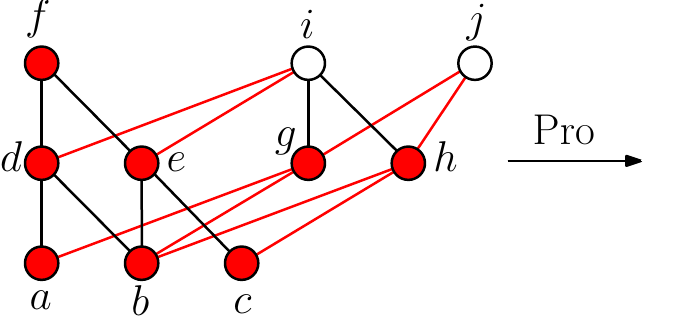}
	\includegraphics[width=.18\linewidth]{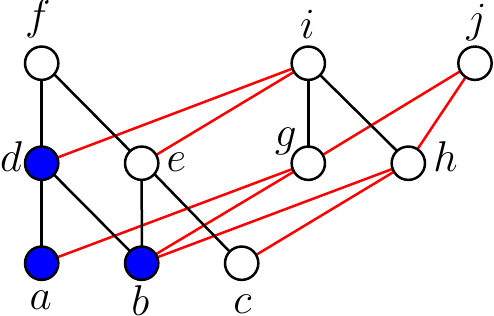}
	\caption{Applying promotion to the red order ideal gives us the blue order ideal.}
	\label{fig:asmproorb}
\end{figure}

\section*{Acknowledgments}
The author thanks his advisor, Jessica Striker, for introducing him to the problem and for helpful discussions along the way. He also would like to thank the anonymous referees for helpful comments. Additionally, he thanks the developers of SageMath \cite{sage} open source mathematical software; SageMath was instrumental in computations. The figures were constructed using Ipe \cite{ipe}.


\bibliographystyle{abbrv}
\bibliography{master}

\end{document}